\documentclass[a4paper]{amsart}
\usepackage[utf8]{inputenc}
\usepackage[margin=3cm]{geometry}
\usepackage{amsmath,amssymb,amsthm}
\usepackage[colorlinks=true,pagebackref=true]{hyperref}
\usepackage{enumerate}
\usepackage{tikz} 
\usepackage{subcaption}
\usepackage[pagebackref=true]{hyperref}
\usepackage{algorithm,algpseudocode}
\usepackage[capitalize]{cleveref}
\usepackage{amsaddr}
\usepackage{comment}

\usepackage{xcolor}

\DeclareMathOperator*{\argmax}{arg\,max}

\newtheorem{theorem}{Theorem}[section]
\newtheorem{lemma}[theorem]{Lemma}
\theoremstyle{definition}
\newtheorem{definition}[theorem]{Definition}
\newtheorem{proposition}[theorem]{Proposition}
\newtheorem{remark}[theorem]{Remark}

\newtheorem{assumption}[theorem]{Assumption}
\newtheorem*{example}{Example}
\numberwithin{equation}{section}

\newcommand{\IR}{\mathbb R}
\newcommand{\by}{\boldsymbol y}
\newcommand{\bz}{\boldsymbol  z}
\newcommand{\bs}{\boldsymbol s}

\newcommand{\todo}{\color{red}}

\title{Bayesian inversion for { Electrical Impedance Tomography} by sparse interpolation}
\author{Quang Huy Pham \and Viet Ha Hoang}
\address{Division of Mathematical Sciences,\\
	School of Physical and Mathematical Sciences,\\
	Nanyang Technological University\\
         21 Nanyang link,	Singapore 637371}

\begin{document}
	
	\maketitle
	
	\begin{abstract}
%		We study the Bayesian approach to EIT, regarding it as an inverse problem of making inferences about the conductivity of a body given electrode measurements of voltage on its surface. We consider the setting where the conductivity is parametrized by a vector whose components are bounded and of decaying importance. We use multivariate Lagrange interpolation in an anisotropic polynomial space to approximate the forward map and employ the interpolant as a surrogate. We propose an algorithm coupling the surrogate with a dimension-robust MCMC method. We analyze the complexity and use numerical examples to demonstrate the possibility of accelerating MCMC using our proposed algorithm. 

We study the Electrical Impedance Tomography Bayesian inverse problem for recovering the conductivity given noisy measurements of the voltage on some boundary surface electrodes. The uncertain conductivity depends linearly on a countable number of uniformly distributed random parameters in a compact interval, with the coefficient functions in the linear expansion decaying with an algebraic rate. We analyze the surrogate Markov Chain Monte Carlo (MCMC) approach for sampling the posterior probability measure, where the multivariate sparse adaptive interpolation, with interpolating points chosen according to a lower index set,  is used for approximating the forward map. The forward equation is approximated once  before running the MCMC for all the realizations, using interpolation on the finite element (FE) approximation at the parametric interpolating points. When evaluation of the solution is needed for a realization, we only need to compute a polynomial, thus cutting drastically the computation time. We contribute a rigorous error estimate for the MCMC convergence. In particular, we show that there is a nested sequence of interpolating lower index sets  for which we can derive an interpolation error estimate in terms of the cardinality of these sets, uniformly for all the parameter realizations. An explicit convergence rate for the MCMC sampling of the posterior expectation of the conductivity is rigorously derived, in terms of the interpolating point number, the accuracy of the   FE approximation of the forward equation, and the MCMC sample number. However, a constructive algorithm for identifying this nested sequence of lower sets has not been available in the literature. We perform numerical experiments using an adaptive greedy approach to construct the sets of interpolation points. The experiments show that adaptive sparse interpolation MCMC recovers the conductivity with an equal, perhaps even better, quality than that produced by the simple MCMC procedure where the forward equation is repeatedly solved for all the samples, but using an exponentially less computation time. Although we study theoretically the case where the conductivity depends linearly on the random parameters, our numerical results indicate that the method works equally for log-uniform conductivities, whose natural logarithms depend linearly on these parameters.  We also demonstrate that MCMC using an adaptively constructed set of interpolation points produces far better results than those when this set being chosen isotropically treating all the random parameters equally.

	\end{abstract}
	\section{Introduction} 

Electrical Impedance Tomography (EIT) is a non-invasive imaging technique where the conductivity of an object is inferred through measurements on electrodes attached to its surface. Applications of  EIT inverse problems are far reaching and important. We mention exemplarily the papers \cite{Adler2012, Brown2003, Cheney1999} and references therein. Mathematically, EIT is  regarded as a highly ill-posed nonlinear inverse problem (\cite{Muller2012}).  Regularization techniques for the EIT problem have been extensively studied (see, e.g., \cite{Gehre2014, Lechleiter2006}). The Bayesian approach to inverse problems (\cite{Kaipio2005, Stuart2010}) has also been employed  for the EIT problem, though to a less extent, see, e.g., \cite{ Dunlop2016, Hyvonen2015, Kaipio2000}. It assumes that the measurement noise follows a known probability, and the desired conductivity belongs to a prior probability space which represents the a priori known information and beliefs. The solution to a Bayesian inverse problem is the posterior probability, which is the conditional probability of the conductivity given the noisy observation. The problem is well posed;  as long as the forward map is measurable, the problem has a unique solution (\cite{Cotter2009}). As the conductivity belongs to a high (infinite) dimensional probability space, sampling the posterior probability via Markov Chain Monte Carlo (MCMC) is normally an enormously   	complex process. A straightforward application of MCMC requires solving the forward problem with equally high accuracy for a large number of different realizations of the forward equation, leading to a very high computation cost. Reducing the cost of sampling the posterior probability measure for Bayesian inverse problems for forward partial differential equations is a very active research at the moment. The aim of this paper is to contribute a method which reduces drastically the MCMC sampling cost for Bayesian inverse problems for the EIT equation. 
	
We focus on a surrogate method where the solution of the forward equation is approximated by a polynomial, obtained from sparse interpolation of the random variable coordinates that determine the prior probability measure. This polynomial approximation is computed once before running the MCMC; and when evaluation of the solution at a particular random sample of the conductivity is needed, instead of solving the forward equation at this sample, we only need to compute the value of a polynomial, which is far less complicated. 

Surrogate methods for MCMC sampling of the posterior probability have been considered before. In the generalized polynomial chaos (gpc) approach as considered in \cite{Marzouk2007}, the prior uncertainty is propagated through the forward model, where an approximating equation in terms of the physical variables and the stochastic variables, using a finite number of gpc modes, is solved. For approximating the forward partial differential equation, finite elements (FEs) can be used to approximate the chosen gpc coefficient functions. This results in solving a linear system that incorporates all the degrees of freedom for approximating  the gpc coefficient functions. A detailed analysis can be found in, e.g., \cite{Cohen2010, Cohen2011}. For the EIT Bayesian inverse problems, this approach has been computationally investigated in \cite{Hakula2014, Hyvonen2015, Leinonen2014}. As shown in \cite{Hoang2013} for the case of elliptic forward equations, when the gpc modes are properly chosen so that a convergence rate for the gpc approximation of the forward equation (in the Lebesgue space $L^2$ with respect to the distribution of the stochastic variables) can be derived, the same convergence order in terms of the number of gpc modes chosen for the forward solver is realized in the total MCMC convergence rate.

Collocation methods are another important tool to approximate forward stochastic/parametric equations. While gpc-based methods are intrusive, i.e. they assume prior knowledge on the distribution of the random parameters in the equations' coefficients, and polynomial basis of the Lebesgue space $L^2$ of the random parameters with respect to the prior probability measure are needed (which may need to be established numerically if the probability density is complicated), collocation method is not intrusive.  It only needs to solve the forward equation at each parametric collocation point separately, which can be done in parallel. Given a collocation polynomial approximation for the forward equation that is uniform for all the parameters (i.e. an approximation in the $L^\infty$ parameter space) such as the sparse interpolation approximation in this paper, the approximation may be used as a surrogate model for MCMC sampling for Bayesian inverse problems with any priors. If a convergence rate in terms of the number of collocation points for the forward solver is available, an MCMC convergence rate can be shown. Further, an apparent advantage of the collocation approximation approach is that when better accuracy is desired by adding additional collocation points, we only need to solve the forward equation at these points, unlike in the gpc FE approach where the whole linear system encoding the interaction of all the chosen gpc modes needs to be solved. A collocation-based surrogate model for MCMC sampling of the posterior for Bayesian inverse problems is considered in \cite{Marzouk2009}. However, to the best of our knowledge, a rigorous error analysis for a collocation surrogate MCMC method for Bayesian inverse problems has not been performed. Perhaps, it is due to the fact that explicit convergence rates for collocation forward solvers have only been established for the sparse interpolation method in Chkifa et al. \cite{Chkifa2013} as indicated by these authors. This approximation of \cite{Chkifa2013} has not been employed as a surrogate model for Bayesian inverse problems. In this paper, we employ the sparse interpolation approximation of \cite{Chkifa2013} as the surrogate model for the EIT Bayesian inverse problems for inferring the conductivity.  We contribute a rigorous error analysis for the MCMC sampling procedure. As for the coefficient of the random elliptic equation considered in \cite{Chkifa2013}, we assume that the uncertain conductivity is expressed as a linear expansion of a countable number of random variables which are uniformly distributed in the compact interval $[-1,1]$. The coefficient functions of this expansion follow a decay rate in the sup norm. We show that the sparse interpolation convergence results  established in \cite{Chkifa2013} for solving the forward elliptic problems hold for the  forward EIT equation. In particular, under certain assumptions on the decaying rate in the conductivity's expansion in terms of the uniform random variables, there exists a nested sequence of lower sets of interpolation points so that the error estimate for the interpolation approximation  decays as a polynomial order of the cardinality of these sets. Given this explicit rate of convergence, %for the collocation method such as those established in \cite{Chkifa2013}, 
we show the convergence of the sparse interpolation surrogate MCMC approximation. The error is expressed in terms of the cardinality of the set of interpolation points, the accuracy  of the FE approximation for the forward EIT equation, and the number of MCMC samples. 

For the sparse interpolation collocation method, the interpolation nodes may be chosen adaptively in a straightforward and unambiguous manner. To obtain an explicit convergence rate for the surrogate MCMC method, we need an explicit convergence rate for the approximation of the stochastic/parametric forward equation. The existence results of a set of gpc modes and interpolation points for such a rate to exist in \cite{Cohen2010, Cohen2011} and \cite{Chkifa2013} do not provide a constructive manner to identify these sets; this is still an open research problem. In this paper, following the literature, e.g, \cite{Chkifa2013, Gerstner2003}, we choose the set of interpolation points adaptively. This is a simple and unambiguous procedure, where a new interpolation node added so that the set of interpolation points remains a lower set is very likely to reduce the approximation error significantly. This is different from the surrogate procedures considered in the literature, e.g., \cite{Marzouk2009,Hakula2014,Hyvonen2015}, where the gpc modes are chosen isotropically, basing on a uniform upper bound for the polynomial degree of the gpc terms. When the coefficient functions in the expansion of the conductivity follow a decaying rate (which is typically the case for the Karhunen-Loeve expansion for a random variable with a sufficiently smooth covariance), the adaptive approach may produce more accurate results than the approach of choosing the interpolation nodes isotropically by following a simple bound for the polynomial degree, as we demonstrate in the numerical experiment section. 

The outline of the paper is as follows. In section \ref{sect:Bayesian_formulation}, we formulate EIT as a Bayesian inverse problem.  In section \ref{sect:finite_truncation}, we approximate the posterior using finite-dimensional truncation of the conductivity and finite element solution of the resulting truncated EIT forward equation.  Section \ref{sect:Lagrange} reviews the multivariate Lagrange interpolation and an adaptive algorithm introduced in \cite{Chkifa2013} to construct polynomial spaces. We then use the sparse interpolation approximation for the forward equation to construct an approximation for the posterior measure.  In section \ref{sect:MH_MCMC}, we sample the posterior probability by MCMC, using the adaptive sparse interpolation as a surrogate. We use dimension-robust MCMC methods, in particular, the independence sampler and the Reflection Random Walk Metropolis. We present a rigorous result on the convergence of the interpolation surrogate MCMC. When the set of interpolation points is chosen so that an explicit interpolation convergence rate in terms of the cardinality of this set is available, we present a rigorous convergence rate for the sparse interpolation surrogate MCMC.  Numerical examples are presented in section \ref{sect:numeric}. We first consider a conductivity which is expressed as a linear combination of Haar wavelet functions, following   the coefficient of the elliptic equation considered in \cite{Chkifa2013}. We experiment on the case of a reference binary conductivity. The numerical results show that adaptive sparse interpolation surrogate MCMC can recover the binary conductivity with an equivalent quality to that for a plain MCMC process with the same number of MCMC samples, requiring an exponentially less computation time. Though the theory is developed for conductivities which depend linearly on uniform random variables, the sparse interpolation surrogate method equally works in the case of a non-linear dependence. We investigate numerically the log-uniform case, where the natural logarithm of the conductivity depends on the uniformly distributed random variables linearly. The numerical result shows that for the log-uniform conductivity, the adaptive sparse interpolation MCMC recovers the conductivity with a high quality.  Proofs to some theoretical results are presented in  the appendices.

Finally, we mention that although we  consider the forward EIT problem in this paper, the sparse interpolation surrogate MCMC in this paper works for other equations. Furthermore, the case where the coefficients depend on Gaussian random variables and may not be uniformly bounded above and away from zero can be treated equally effectively by sparse polynomial interpolation, and is the subject of our forthcoming works. 

Throughout the paper, by $c$ and $C$, we denote a generic constant whose values may differ between different appearances. Repeated indices indicate summation.

	\section{Bayesian formulation of EIT}\label{sect:Bayesian_formulation}
	\subsection{The forward model}\label{subsect:forward}
We use the smoothened complete electrode model (smoothened CEM), proposed in \cite{Hyvonen2017}, as the forward model.  Let $D$ be a Lipschitz domain in $\mathbb{R}^d$ ($d=2,3$). We attach $K$ electrodes $\{E_k\}_{k=1}^{K}$, which are separated, nonempty, open, connected subsets of the boundary $\partial D$. We denote  by $E=\cup_{k=1}^{K}E_k$. 
	The conductivity $\sigma\in L^\infty(D)$ inside $D$ is assumed to satisfy 
	\begin{equation}\label{sigma_bounded_assumption}
		\sigma^- \le \sigma(x) \le \sigma^+\; \forall x\in D
	\end{equation}
	for some $\sigma^-, \sigma^+ \in (0,\infty)$. The contact admittance $\zeta\in L^\infty(\partial D)$ is such that 
	\begin{equation}\label{zeta_assumptions}
		\zeta\ge 0, \zeta|_{\partial D\backslash \bar{E}}\equiv 0, \zeta|_{E_k}\not\equiv 0\;\forall k.
	\end{equation}
	It follows from \eqref{zeta_assumptions} that there exists open subsets $e_k\subset E_k, k=1,\dots, K$, and $\zeta^->0$ such that 
	\begin{equation}\label{zeta_lower_bound}
		\zeta\ge \zeta^-\;a.e\;\text{on}\;\cup_{k=1}^{K}e_k.
	\end{equation}
	For each $k=1,\dots, K$, a current $I_k$ is applied to the electrode $E_k$. The current pattern is represented by a vector $I=(I_1,\dots,I_K)$ in 
	\begin{equation*}
		\mathbb{R}^K_{\diamond}:=\{U\in \mathbb{R}^K:U_1+\dots+U_K=0\}. 
	\end{equation*}
	The voltages $V_1,\dots,V_K$ at the corresponding electrodes are measured. The voltage vector $V=(V_1,\dots,V_K)\in \mathbb{R}^K$ is associated with a piecewise constant function $V:\partial D\to \mathbb{R}$ defined by 
	\begin{equation*}
		V:=\sum_{k=1}^K V_k1_{E_k},
	\end{equation*} 
	where $1_{E_K}$ denotes the characteristic function of the set $E_k$. 
	It should be clear from the context whether $V$ denotes a vector or a function on the boundary. 
	
	%According to the smoothened CEM
	The potential $v:D\to \mathbb{R}$ inside $D$ and the voltage vector $V$ on the electrodes satisfy the following elliptic PDE
	\begin{equation}\label{SCEM}
		\begin{cases}
			-\nabla \cdot (\sigma \nabla v)=0\quad \text{in } D,\\
			\int_{E_k}\nu \cdot \sigma \nabla v dS=I_k,\quad k=1,\dots,K, \\
			\nu \cdot \sigma \nabla v  =\zeta(V-v)\quad \text{on }\partial D,
		\end{cases}
	\end{equation}
	where $\nu \in L^\infty(\partial D,\mathbb{R}^d)$ is the exterior unit normal of $\partial D$. 
	%The contact admittance function $\zeta$ is positive inside each $E_k$ and is zero outside.
	When the contact admittance is constant on the electrodes and $0$ elsewhere, i.e., $\zeta = \sum_{k=1}^K \frac{1}{z_k} 1_{E_k}$, where $z_k$ are constants, the problem becomes the standard CEM problem (see, e.g., \cite{Somersalo1992}). We consider the model of a contact admittance $\zeta$ to facilitate the $H^2(D)$ regularity of the solution $v$ of the forward equation, which enables the optimal convergence rate of the FE approximation of the forward equation in the analysis when $\zeta$ is sufficiently regular. However, the sparse interpolation surrogate sampling procedure we develop in this paper works for the general case, albeit with a lower convergence rate for the forward FE solver when $\zeta$ is not sufficiently regular.

	In \cite{Hyvonen2017}, following the reasoning in \cite{Somersalo1992},  the boundary value problem \eqref{SCEM} is formulated in variational form as follows. 
	Let ${\bf 1}=(1,1,\ldots,1)^{\top}\in\mathbb{R}^K$. We consider the function space 
	\begin{equation*}
		\mathcal{H}^1:=(H^1(D)\oplus \mathbb{R}^K)/\mathbb{R},
	\end{equation*}
	where $(v,V)\in H^1(D)\oplus\IR^K$ and $(v+c, V+c{\bf 1})$ are considered as being in the same equivalence class. The space $\mathcal{H}^1 $ is equipped with the norm
	\[
	\|(v,V)\|_{\mathcal{H}^1}=\inf_{c\in\mathbb{R}}\|v-c\|_{H^1(D)}+|V-c{\bf 1}|
	\]
	where $|\cdot|$ denotes the norm in the Euclidean space $\mathbb{R}^{K}$.
	Then $(v,V)\in \mathcal{H}^1$ is a weak solution to \eqref{SCEM} if and only if 
	\begin{equation}\label{variational}
		B((v,V),(w,W);\sigma)=I\cdot W\quad \forall\;(w,W)\in \mathcal{H}^1, 
	\end{equation}
	where the bilinear form $B:\mathcal{H}^1\times \mathcal{H}^1\to \mathbb{R}$ is defined as 
	\begin{equation*}
		B((v,V),(w,W);\sigma):=\int_D \sigma \nabla v\cdot \nabla w dx+\int_{\partial D}\zeta (V-v)(W-w)dS.
	\end{equation*}
	Under assumptions \eqref{sigma_bounded_assumption} and \eqref{zeta_assumptions}, the bilinear form $B$ is bounded and coercive (see \cite{Hyvonen2017}). 	Let $\zeta^+:=||\zeta||_{L^\infty(\partial D)}$. Examining the proof of Lemma 2.1 in \cite{Hyvonen2017}, we deduce 
	\begin{equation*}
		|B((v,V),(w,W))|\le C\max\{\sigma^+,\zeta^+\}||(v,V)||_{\mathcal{H}^1}||(w,W)||_{\mathcal{H}^1}
	\end{equation*}
	and 
	\begin{equation*}
		B((v,V),(v,V))\ge C\min\{\sigma^-,\zeta^-\}||(v,V)||_{\mathcal{H}^1}^2, 
	\end{equation*}
	for all $(v,V)$ and $(w,W)$ in ${\mathcal H}^1$. 
	Hence, by the Lax-Milgram lemma, there exists a unique solution $(v,V)\in \mathcal{H}^1$ to the variational problem \eqref{variational}  which satisfies
	\begin{equation}\label{H1_norm_bound}
		||(v,V)||_{\mathcal{H}^1}\le C\frac{|I|}{\min\{\sigma^-, \zeta^-\}}.
	\end{equation} 
	
	The choice of a representative in the quotient space corresponds to the choice of a ground or reference voltage.  In the rest of the paper, without loss of generality, we choose the representative $(v,V)$ of the equivalence class such that 
	\begin{equation} \label{sumV=0}
		V_1+\dots+V_K=0.
	\end{equation}
	Thus, given a conductivity $\sigma$ such that \eqref{sigma_bounded_assumption} holds, each current pattern $I\in \mathbb{R}^K_{\diamond}$ gives rise to a voltage vector $V=V(\sigma)\in \mathbb{R}^K_{\diamond}$. 
	The relation between current patterns and voltage vectors can be described by a linear mapping $R(\sigma):\mathbb{R}^K_{\diamond}\to \mathbb{R}^K_{\diamond}$, defined by 
	\begin{equation*}
		R(\sigma)I:=V(\sigma).  
	\end{equation*}

	\begin{lemma}\label{V<=(v,V)}
		For $(v,V)\in \mathcal{H}^1$ satisfying $V\in \mathbb{R}^K_\diamond$,   
		\begin{equation*}
			|V|\le ||(v,V)||_{\mathcal{H}^1}. 
		\end{equation*}
	\end{lemma} 
	\begin{proof}
		We have $||(v,V)||_{\mathcal{H}^1}\ge \inf_{c\in \mathbb{R}}|V-c|.$ The quadratic function $c\mapsto |V-c|^2$ attains its minimum at $c=\frac{V_1+\dots+V_K}{K}=0$.
	\end{proof}
	
	%In Appendix \ref{appendix:complex_SCEM}, to establish a bound for the coefficient of the Taylor expansion of the solution of the parametric problem with respect to the parameters, we consider the general setting where the conductivity and current pattern can be complex-valued. 

	\subsection{The inverse problem}
	We apply $K-1$ linearly independent current patterns $I^{(1)},\dots,I^{(K-1)}\in \mathbb{R}^K_{\diamond}$. The given data are noisy measurements of the corresponding voltage vectors $V^{(k)}(\sigma):=R(\sigma)I^{(k)}, k=1.\dots,K-1$. We concatenate the $K-1$ voltage vectors to form the forward map $\mathcal{V}(\sigma):=(V^{(1)}(\sigma),\dots,V^{(K-1)}(\sigma))$. We assume that the measurement noise is centered Gaussian. That is, we assume that the observed data is 
	\begin{equation}\label{data_sigma}
		\delta = \mathcal{V}(\sigma) + \vartheta, 
	\end{equation}
	with $\vartheta\sim N(0,\Sigma)$, where $\Sigma$ is a covariance matrix of size $K(K-1)\times K(K-1)$. 
	
	The inverse problem is to reconstruct the conductivity $\sigma$ from the data $\delta$. %This problem is well-known to be severely ill-posed.  
	
	\subsection{Parametrization}
	We assume that the conductivity $\sigma$ is parametrized by a sequence $(y_j)_{j\ge 1}\subset [-1,1]$. In other words, we assume that 
	\begin{equation*}
		\sigma(x)=\sigma(x,\by),
	\end{equation*}
	where $\by=(y_1, y_2, \ldots)\in [-1,1]^{\mathbb{N}}$ is an unknown parameter. 
	We consider the parametric conductivity of the form 
	\begin{equation}\label{affine_para}
		\sigma(x,\by)= \bar{\sigma}(x)+\sum_{j=1}^{\infty}y_j\psi_j(x), \quad  \by=(y_1,y_2,\ldots)\in [-1,1]^\mathbb{N},
	\end{equation}
	where $\bar{\sigma}, \psi_j$ are real-valued functions in $ L^\infty(D)$. 
	Though we only develop the sparse interpolation MCMC approach for  conductivities of the form \eqref{affine_para}, in the numerical experiment section, we show that the sparse interpolation MCMC method works equally well for the case of log-uniform conductivities, i.e. when the conductivity is of the form
	\begin{equation}\label{log_affine_para}
		\sigma(x,\by)= \exp\left(\bar{\sigma}(x)+\sum_{j=1}^{\infty}y_j\psi_j(x)\right), \quad \by\in [-1,1]^\mathbb{N}. 
	\end{equation}
	
	%Such parametrizations can be obtained from the Karhunen-Loeve expansion for stochastic processes or by randomizing coefficients in well-known expansions for functions, such as Fourier, wavelet.  
	
	To ensure uniform ellipticity with respect to all the parameter sequences $\by$, we assume that there exist constants $\sigma^-$ and $\sigma^+$ independent of $\by$ such that 
	\begin{equation}\label{elliptic_parametric}
		0<\sigma^-\le \sigma(x,\by)\le \sigma^+<\infty\; \forall x\in D, \by\in [-1,1]^\mathbb{N}. 
	\end{equation}
	We note that \eqref{elliptic_parametric} is equivalent to the following condition: 
	\begin{equation}\label{sum_psi}
		\sum_{j=1}^{\infty}|\psi_j(x)|\le \min\{\bar{\sigma}(x)-\sigma^-,\sigma^+ -\bar{\sigma}(x) \}\; \forall x\in D. 
	\end{equation}
	Then $\mathcal{V}(\sigma(\cdot,\by))$ is well-defined for every $\by\in [-1,1]^\mathbb{N}$, and is uniformly bounded with respect to all $\by\in [-1,1]^{\mathbb N}$.

	\subsection{Bayesian inversion}
	With the parametrization $\sigma(x)=\sigma(x,\by)$, inferences about the conductivty $\sigma$ is equivalent to inferences about $\by$. For $\by\in [-1,1]^\mathbb{N}$, let the forward map be
	\begin{equation*}
		\mathcal{G}(\by):=\mathcal{V}(\sigma(\cdot,\by)). 
	\end{equation*}
	%The map $\mathcal{G}$ is called the forward map. 
	The data in \eqref{data_sigma} can be rewritten as 
	\begin{equation*}
		\delta = \mathcal{G}(\by)+\vartheta,\quad \vartheta\sim N(0,\Sigma). 
	\end{equation*}
	%In the Bayesian approach, we treat the unknown $y\in [-1,1]^\mathbb{N}$ as a random variable taking values in $[-1,1]^\mathbb{N}$. 
	Let $U=[-1,1]^{\mathbb N}$; $U$ is equipped with the tensor product $\sigma$ algebra
	\[
	\Theta=\bigotimes_{j=1}^\infty {\mathcal B}(\IR),
	\]
	where ${\mathcal B}(\IR)$ is the Borel sigma algebra on $\IR$.
	We assume that the $y_j$ ($j\in\mathbb{N}$) are uniformly distributed in $[-1,1]$. The prior probability measure in $(U,\Theta)$ is defined as 
	\begin{equation}
		d\mu_0(\by)=\bigotimes_{j=1}^\infty{dy_j\over 2}.
		\label{eq:mu0}
	\end{equation}
	%We place a prior measure $\mu_0$ on the parameter space $[-1,1]^\mathbb{N}$. We are interested in the 
	Our aim is to determine the posterior probability measure $\mu^\delta=\mathbb{P}(\by|\delta)$. % of the state $y$ given data $\delta$. 
	%Adapting the proofs in \cite{Dunlop2016}, we obtain the existence of the posterior $\mu^\delta$ and the well-posedness of the Bayesian inverse problem being considered. 
	%We have the following existence result. 
	We define the misfit function as
	\begin{equation*}
			\Phi(\by;\delta):=\frac{1}{2}|\delta-\mathcal{G}(\by)|_\Sigma^2 = \frac{1}{2}|\Sigma^{-1/2}(\delta-\mathcal{G}(y))|^2.
	\end{equation*}

	We have the following existence result.
	\begin{proposition}\label{exist_post_uni}
		The posterior $\mu^\delta$ is absolutely continuous with respect to $\mu_0$ and satisfies 
		\begin{equation}
			\frac{d\mu^\delta}{d\mu_0}(\by) \propto \exp\left(-\Phi(\by;\delta)\right).
		\label{eq:mudelta}
		\end{equation}
%		where $\Phi$ is the misfit function defined by
%		\begin{equation*}
%			\Phi(\by;\delta):=\frac{1}{2}|\delta-\mathcal{G}(\by)|_\Sigma^2 = \frac{1}{2}|\Sigma^{-1/2}(\delta-\mathcal{G}(y))|^2.
%		\end{equation*}
	\end{proposition}
	%{\todo Huy: at least mention something on the measurability of the forward map, or better supply a short proof}

	\begin{proof}	
	%We denote by $\mathcal{A}(D)$ the class of $\sigma\in L^\infty(D)$ such that \eqref{elliptic_parametric} is satisfied for some $\sigma^+,\sigma^->0$. On $\mathcal{A}(D)$, we consider the norm of $L^{\infty}(D)$. Similar to Proposition 2.7 in \cite{Dunlop2016}, 
	%we obtain in Lemma \ref{M1,2} the continuous dependence of the forward solution $(v(\sigma),V(\sigma))$ on the conductivity $\sigma$.   
	%It follows from Lemma \ref{V<=(v,V)} that the map $\sigma\mapsto V(\sigma)$ is also continuous. Since $\by\mapsto \sigma(\by)$ is continuous, the forward map $\mathcal{G}$ is continuous, hence measurable. 
	 From Lemma \ref{M1,2}, $L^\infty(D)\ni\sigma\mapsto (v(\sigma),V(\sigma))\in \mathcal{H}^1$ is continuous. Since $U\ni\by\mapsto\sigma(\by)\in L^\infty(D)$ is continuous,  the forward map $\mathcal{G}$, as a map from $U$ to ${\mathbb R}^K$ is continuous, hence measurable. From Theorem 2.1 in \cite{Cotter2009}, we get the conclusion.
	
	%By applying the Bayes theorem for functions (see Theorem 2.1 in \cite{Cotter2009}), we get the conclusion.
	\end{proof} 
	
	For two measures $\mu$ and $\mu'$ which are both absolutely continuous with respect to a measure $\mu_0$, the Hellinger distance between them is defined as 
	\begin{equation*}
		d_{Hell}(\mu,\mu'):=\sqrt{\left(\frac{1}{2}\int \left(\sqrt{\frac{d\mu}{d\mu_0}}-\sqrt{\frac{d\mu'}{d\mu_0}}\right)^2d\mu_0\right)}.
	\end{equation*} 
	We have the following result on the well-posedness of the posterior measure.
	\begin{proposition}\label{well_posed}
		For every $r>0$ and $\delta,\delta'$ such that $|\delta|_\Sigma,|\delta'|_\Sigma\le r$,
		\begin{equation*}
			d_{Hell}(\mu^\delta,\mu^{\delta'})\le C(r)|\delta-\delta'|_\Sigma.
		\end{equation*}
	\end{proposition}
	
	%{\todo Huy: at least mention where the proof can be found, or better sketch briefly the proof}
	\begin{proof}
		From \eqref{H1_norm_bound} and Lemma \ref{V<=(v,V)}, there exists a constant $C$ such that 
		\begin{equation}\label{G_bounded}
			\sup_{\by\in U} |\mathcal{G}(\by)|_\Sigma \le C<\infty. 
		\end{equation} 
		Hence for all $\by\in U$ and $|\delta|_\Sigma\le r$, 
		\begin{equation}\label{Phi_bounded}
			\Phi(y;\delta) \le \frac{1}{2}(r+C)^2.
		\end{equation}
		From the definition of $\Phi$, we have 
		\begin{align*}
			|\Phi(\by;\delta)-\Phi(\by;\delta')|&=\frac{1}{2}|\langle \delta+\delta'-2\mathcal{G}(\by),\delta-\delta' \rangle_\Sigma|\\
			&\le \frac{1}{2}(|\delta|+|\delta'|+2|\mathcal{G}(\by)|_\Sigma)|\delta-\delta'|_\Sigma. 
		\end{align*}
		Hence for every $y\in U$ and $\delta, \delta'$ such that $|\delta|_\Sigma, |\delta'|_\Sigma\le r$, 
		\begin{equation}\label{Phi_delta}
			|\Phi(\by;\delta)-\Phi(\by;\delta')|\le (r+C)|\delta-\delta'|. 
		\end{equation}
		With \eqref{Phi_bounded} and \eqref{Phi_delta}, the assertion then follows from Theorem 2.4 in \cite{Hoang2012}. 
	\end{proof}
	\section{Approximation of the posterior by finite truncation and finite elements}\label{sect:finite_truncation}
	\subsection{Finite-dimensional truncation}
	In numerical computation, we need to consider a finite number of parameters in expansion \eqref{affine_para}. For $\by\in [-1,1]^{\mathbb{N}}, J\in \mathbb{N}$, let $\by^J:=(y_1,\dots,y_J)$. We approximate $\sigma(x,\by)$ by 
	\begin{equation*}
		\sigma^J(x,\by):=\sigma(x,(y_1,\ldots,y_J,0,0,\dots)). 
	\end{equation*}
	The parametrization for $\sigma^J$ only depends on a vector with finite length $J$. 
	From \eqref{elliptic_parametric}, it follows that for all $J\in \mathbb{N}$: 
	\begin{equation}\label{sigma_J_bounded}
		0<\sigma^-\le \sigma^J(x,\by)\le \sigma^+<\infty\; \forall \by\in [-1,1]^ \mathbb{N}. 
	\end{equation}
	%Hence for each $y^J\in [-1,1]^J$, 
	We consider the approximation to the forward map $\mathcal{G}$ given by the truncated forward map
	\begin{equation*}
		\mathcal{G}^J(\by):=\mathcal{V}(\sigma^J(\cdot,\by)).
	\end{equation*}
	%is well-defined. 
	We then approximate the posterior measure $\mu^\delta$ by the measure $\mu^{J,\delta}$, whose density with respect to the prior $\mu_0$ is defined by
	\begin{equation}
		\frac{d\mu^{J,\delta}}{d\mu_0}(\by)\propto \exp(-\Phi^J(\by;\delta)),
		\label{eq:muJ}
	\end{equation}
	where 
	\begin{equation}
		\Phi^J(\by;\delta):=\frac{1}{2}|\delta-\mathcal{G}^J(\by)|_\Sigma^2.
		\label{eq:PhiJ}
	\end{equation}
		To estimate the distance between $\mu^{J,\delta}$ and $\mu^\delta$, we impose the following assumption about the decay rate of $(||\psi_j||_{L^\infty(D)})_{j\ge 1}$. 
	\begin{assumption}\label{lp-summability_assumption}
		There exist constants $s, C>0$ such that for all $j\in \mathbb{N}$, 
		\begin{equation*}
			||\psi_j||_{L^\infty(D)}\le Cj^{-(1+s)}.  
		\end{equation*}
	\end{assumption}  

	\begin{proposition}\label{mu-muJ}
		\begin{comment}
			\todo{Insert the result
				\[
				d_{Hell}(\mu^\delta,\mu^{J,\delta})\le cJ^{-s}
				\]
				here. Mention briefly how to prove it.
			} 
		\end{comment}
		
	Under Assumption \ref{lp-summability_assumption}, 
	\begin{equation*}
			d_{Hell}(\mu^\delta,\mu^{J,\delta})\le CJ^{-s}. 
	\end{equation*}
	\end{proposition}
	\begin{proof}
		From Assumption \ref{lp-summability_assumption}, for every $x\in D, \by\in U$ we have 
	\begin{equation}\label{sigma-sigmaJ}
		||\sigma(\cdot,\by)-\sigma^J(\cdot,\by)||_{L^\infty(D)}\le \sum_{j>J}||\psi_j||_{L^\infty(D)}\le C\int_J^{\infty} t^{-(1+s)}dt = CJ^{-s}. 
	\end{equation}
	Since $\sigma^J(x,\by)\ge \sigma^-$, it follows from \eqref{M1-M2} and Lemma \ref{V<=(v,V)} that
	\begin{equation*}
		|V(\sigma(\by))-V(\sigma^J( \by))|\le C||\sigma(\cdot,\by)-\sigma^J(\cdot,\by)||_{L^\infty(D)}. 
	\end{equation*}
	 Hence
	\begin{equation*}
		\sup_{\by\in U}|\mathcal{G}(\by)-\mathcal{G}^J(\by)|_\Sigma\le CJ^{-s}. 
	\end{equation*}
	From the definitions of $\Phi$ and $\Phi^J$, we deduce  
	\begin{align*}
		\sup_{\by\in U}|\Phi(\by)-\Phi^J(\by)|& = \frac{1}{2}|\langle \mathcal{G}( \by)-\mathcal{G}^J( \by), 2\delta + \mathcal{G}( \by) + \mathcal{G}^J(\by) \rangle_\Sigma|\\
		&\le \frac{1}{2}|\mathcal{G}(\by)-\mathcal{G}^J(\by)|_\Sigma (2|\delta|_\Sigma + |\mathcal{G}( \by)|_\Sigma + |\mathcal{G}^J( \by)|_\Sigma). 
	\end{align*}
	Similar to \eqref{G_bounded}, since $\sigma^J(\by)\ge \sigma^-$ for all $J\in \mathbb{N}, \by\in U$, it follows from \eqref{H1_norm_bound} and Lemma \ref{V<=(v,V)} that 
	\begin{equation}\label{GJ_bounded}
		\sup_{\by\in U, J\in \mathbb{N}}|\mathcal{G}^J(\by)|<\infty. 
	\end{equation} 
	As a result, 
	\begin{equation}\label{phi-phi_J}
		\sup_{\by\in U}|\Phi(\by)-\Phi^J(\by)|\le CJ^{-s}.
	\end{equation}
 	With \eqref{phi-phi_J}, \eqref{G_bounded} and \eqref{GJ_bounded}  in hand, the proof is completed by using a standard argument to estimate Hellinger distance (see, e.g., \cite{Cotter2009,Stuart2010}). The details are given in Appendix \ref{section:error_estimate}. 
	\end{proof}
	We may as well regard $\mu^{J,\delta}$ as a measure on the finite-dimensional space $U_J:=[-1,1]^J$. Let $\mu_0^J$ be the projection of $\mu_0$ from $[-1,1]^{\mathbb{N}}$ to $[-1,1]^J$. With some abuse of the notation, we denote $\mu^{J,\delta}$ as the measure in $[-1,1]^J$ with the density given by 
	\begin{equation*}
		\frac{d\mu^{J,\delta}}{d\mu_0^J}(\by^J)\propto \exp(-\Phi^J((y_1,\ldots,y_J,0,0,\ldots);\delta)). 
	\end{equation*}

	\subsection{FE approximation}
		Let $\{\mathcal{T}^l\}_{l\ge 1}$ be a family of shape-regular triangulations of $D$ with mesh-size $h_l=O(2^{-l})$.
	We denote by
	\begin{equation*}
		P^l=\{w\in C^0(\bar{D}): w\big|_{T}\in \mathcal{P}_1(T)\;\forall T\in \mathcal{T}^l \}\subset H^1(D),
	\end{equation*}
	where $\mathcal{P}_1(T)$ is the set of linear polynomials in the simplex $T\in \mathcal{T}^l$.
	We find the FE approximate solutions in the finite-dimensional subspace 
	\begin{equation*}
		\mathcal{V}^l=(P^l\oplus\mathbb{R}^K)/\mathbb{R}.
	\end{equation*}  
		We define the FE approximation of $\mathcal{G}^J$ as follows. Recall that 
	\begin{equation*}
		\mathcal{G}^J(\by)= \mathcal{V}(\sigma^J(\cdot,\by)) = (V^{(k)}(\sigma^J(\cdot,\by)))_{k=1}^{K-1}.
	\end{equation*}  
	For each $\by\in U$, since the bilinear form $B(\cdot,\cdot; \sigma^J(\by))$ is coercive and bounded, there exists a unique solution $(v^{J,l,(k)}(\by),V^{J,l,(k)}(\by))\in \mathcal{V}^l$ to the problem
	\begin{equation*}
		B((v^{J,l,(k)}(\by),V^{J,l,(k)}(\by)),(w^l,W^l);\sigma^J(\by))=I^{(k)}\cdot W^l\;\forall (w^l,W^l)\in \mathcal{V}^l.
	\end{equation*} 
 We approximate the forward map by concatenating these FE solutions into 
	\begin{equation*}
		\mathcal{G}^{J,l}(\by):= (V^{J,l,(k)}(\by))_{k=1}^{K-1}.
	\end{equation*} 
	We then approximate the posterior $\mu^\delta$ by the measure $\mu^{J,l,\delta}$ defined by 
	\begin{equation}
		\frac{d\mu^{J,l,\delta}}{d\mu_0}(\by)\propto \exp(-\Phi^{J,l}(\by;\delta)),
		\label{eq:muJl}
	\end{equation}
	where 
	\begin{equation}
		\Phi^{J,l}(\by;\delta):=\frac{1}{2}|\delta-\mathcal{G}^{J,l}(\by)|^2_{\Sigma}. 
		\label{eq:PhiJl}
	\end{equation}
	We may also view $\mu^{J,l,\delta}$ as a measure on $[-1,1]^J$. 
	
	To estimate the FE error, we make the following assumption about the regularity of $\zeta$ and $(\psi_j)_{j\ge 1}$. 
		\begin{assumption}\label{assumption:sigma,zeta}
		We assume that $D$ is a convex polygon, and each electrode $E_k$ is a proper subset of the interior of one affine edge ($d=2$) or surface ($d=3$) of $\partial D$. Further, $\zeta\in W^{1,\infty}(\partial D)$,
		%$\zeta\in H^t(\partial D)$ for some $t>\frac1{2}$, 
		$\bar{\sigma}\in W^{1,\infty}(D)$ and $ (\psi_j)_{j\ge 1}\subset W^{1,\infty}(D)$ such that
		\begin{equation*}
			\sup_{x\in D}\sum_{j\ge 1}\|\nabla \psi_j(x)\|_{L^{\infty}(D)}<\infty.
		\end{equation*}
	\end{assumption}
	We deduce from Proposition \ref{V-Vl} that if Assumptions \ref{assumption:sigma,zeta} holds, 
	\begin{equation*}
		|\mathcal{G}^{J}(\by)-\mathcal{G}^{J,l}(\by)|\le C2^{-2l}.%(\max\{\sigma^+,\zeta^+\})^3C_1(\sigma(\by),\zeta)^2|I|, 
	\end{equation*}
%	where 
%	$C_1(\sigma(\by),\zeta)=\frac{||\sigma^J(\by)||_{W^{1,\infty}(D)}+\sigma^{-}||\zeta||_{H^s(\partial(D))}||(\sigma^J)^{-1}(\by)||_{W^{1,\infty}(D)}}{\sigma^-\min\{\sigma^-,\zeta^-\}^2}$ {\ha which is uniformly bounded for all $\by\in U$}. 
%	In addition, Assumptions \ref{assumption:sigma,zeta} also implies 
%	\begin{equation*}
%		\sup_{J\in \mathbb{N}}\sup_{\by\in U_J} ||\sigma^J(\by)||_{W^{1,\infty}(D)} <\infty  
%	\end{equation*}
%	and 
%	\begin{equation*}
%		\sup_{J\in \mathbb{N}}\sup_{\by\in U_J} ||(\sigma^J)^{-1}(\by)||_{W^{1,\infty}(D)} <\infty.   
%	\end{equation*}
	As a result, we obtain the error estimate of the FE approximation of the truncated forward map.  
	\begin{proposition}\label{G-Gl}
		Under Assumption \ref{assumption:sigma,zeta}, there exists a constant $C$ which is independent of $l$ such that 
		\begin{equation*}
			\sup_{J\in \mathbb{N}}\sup_{\by\in U_J}|\mathcal{G}^{J}(\by)-\mathcal{G}^{J,l}(\by)|_\Sigma\le C2^{-2l}.
		\end{equation*}
	\end{proposition}
	Analogous to Proposition \ref{mu-muJ}, we have the following estimate.
	\begin{proposition}
		\begin{comment}
			\todo{Insert the result
				\[
				d_{Hell}(\mu^\delta,\mu^{J,l,\delta})\le c(J^{-s}+2^{-2l})
				\]
				here. Mention briefly how to prove it.
			}
		\end{comment}
		
		Under Assumptions \ref{lp-summability_assumption} and \ref{assumption:sigma,zeta}, 
		\begin{equation*}
				d_{Hell}(\mu^\delta,\mu^{J,l,\delta})\le C(J^{-s}+2^{-2l}).
	\end{equation*}
	\end{proposition}

\begin{remark} In Assumption \ref{assumption:sigma,zeta}, we assume that $\zeta\in W^{1,\infty}(D)$ to get the optimal FE convergence rate $2^{-2l}$ for approximating $V$. However, this can be weakended to $\zeta\in H^t(\partial D)$ for $t>(d-1)/2$ (see \cite{Hyvonen2017}). When $\zeta$ is not sufficiently regular such as in the standard CEM model where it is only piecewise constant, the solution $v$ of the forward problem possesses a weaker regularity than $H^2(D)$. We get a weaker FE convergence rate. 
\end{remark}

	\section{Approximation of the posterior by Lagrange interpolation of the forward map}\label{sect:Lagrange}
	We approximate the forward map $\mathcal{G}$ by an interpolant of the multivariate function $\mathcal{G}^{J,l}$  with respect to $\by^J\in U_J=[-1,1]^J$. Therefore, we consider the problem of fitting a function $g:[-1,1]^J\mapsto \mathbb{R}^m$  (in our context $m=K(K-1))$ by a multivariate polynomial that agrees with $g$ at $N$ distinct points in $[-1,1]^J$. In the rest of the paper, for conciseness, instead of using $\by^J$, we use $\by$ to indicate vectors in $[-1,1]^J$. We simplify the presentation in \cite{Chkifa2013}, where $g$ is the solution of parametric  PDEs and the range of $g$  can be in an infinite-dimensional space. 
	%{\todo I think in Chkifa et al they presented interpolation for the case where $g$ is a general function, not really only for the solution of a parametric equation}.
	
	\subsection{Interpolation on sparse grids}
	Interpolation on full-tensor grids suffers from the curse of dimensionality as the number of interpolation points grows exponentially with the dimension $J$. The idea of sparse grids can be traced back to Smolyak \cite{Smolyak1963}. We consider Lagrange interpolation. Multivariate interpolation with other basis functions is studied in \cite{Barthelmann2000,Klimke2005}.

	Let $z_0, z_1, \dots, z_n, \dots$ be a sequence of distinct points in $[-1,1]$. Let $g$ be a univariate function in $ C[-1,1]$. The unique polynomial of degree at most 
	$n$ interpolating $g$ at the points $z_0,\dots,z_n$ is 
	\begin{equation}\label{Lagrange_1d}
		I_ng(z):=\sum_{k=0}^{n}g(z_k)l_k^n(z)\quad \text{for }z\in [-1,1], 
	\end{equation}
	where 
	\begin{equation*}
		l_k^n(z)=\prod_{j=0,j\neq k}^{n}\frac{z-z_j}{z_k-z_j}
	\end{equation*}
	are the univariate Lagrange polynomials satisfying $l_k^n(z_j)=\delta_{kj}$ for $0\le k,j\le n$. 
We define the univariate difference operator $\Delta_k$   as  
	\begin{equation*}
		\Delta_0 = I_0, \Delta_k:=I_k-I_{k-1}\;\text{for }k\ge 1. 
	\end{equation*}
We can then express 
	\begin{equation*}
		I_n = \sum_{k=0}^{n}\Delta_k.
	\end{equation*}
%	where the univariate difference operator $\Delta_k$ is defined recursively as  
%	\begin{equation*}
%		\Delta_0 = I_0, \Delta_k:=I_k-I_{k-1}\;\text{for }k\ge 1. 
%	\end{equation*}
%	
	
	The interpolant operator $I_n$ can be generalized to the multivariate setting as follows. For $\nu\in \mathbb{N}_0^J$, let 
	\begin{equation*}
		\bz_\nu:=(z_{\nu_1},\dots,z_{\nu_J})\in [-1,1]^J. 
	\end{equation*}
	For $g\in C([-1,1]^J,\mathbb{R}^m)$, the tensorized interpolant is defined by 
	\begin{equation*}
		I_\nu g := \sum_{k_1=0}^{\nu_1}\cdots \sum_{k_J=0}^{\nu_J}f(z_{k_1},\dots,z_{k_J}) \cdot (l_{k_1}^{\nu_1}\otimes \cdots \otimes l_{k_J}^{\nu_J}). 
	\end{equation*}
	We can also write 
	\begin{equation*}
		I_\nu = \otimes_{j=1}^{J}I_{\nu_j}. 
	\end{equation*}
	The tensorized difference operator is defined by 
	\begin{equation*}
		\Delta_\nu :=\otimes_{j=1}^{J}\Delta_{\nu_j}.
	\end{equation*}
	The sparse grid interpolant operator at the points $z_\nu$ with $|\nu|_1:=\sum_{j=1}^J \nu_j\le n$ is given by
	\begin{equation}\label{sparse_interp}
		%	I_{S_n}:= 
		\sum_{|\nu|_{1}\le n} \Delta_{\nu}. 
	\end{equation}
	
	\subsection{Interpolation on lower sets}
	We can extend the concept of sparse interpolation to general lower sets.
	% the sparse grids $\{|\nu|_1\le n \}$ to more general index sets. %The expression \eqref{sparse_interp} is valid if 
	%Sparse interpolation can be defined on index sets $\Lambda$  which are `lower set'. 
	%We define lower sets as below.  
	 We first provide the definition of lower sets.
	On $\mathbb{N}_0^J$, we consider the following partial ordering.  For $\nu=(\nu_j)_{j=1}^J$ and $\nu'=(\nu'_j)_{j=1}^J \in \mathbb{N}_0^J$, we denote by
	\begin{equation*}
		\nu' \prec \nu \quad \text{if } \nu'_j\le \nu_j\; \forall j=1,\dots,J. 
	\end{equation*}
	
	\begin{definition}
		An index set $\Lambda\subset \mathbb{N}_0^J$ is called a lower set if one of the following equivalent conditions is satisfied: 
		\begin{enumerate}[i)]
			\item ($\nu\in \Lambda$ and $\nu'\prec \nu$) implies $\nu'\in \Lambda$. 
			\item ($\nu\in \Lambda$ and $\nu_j\ge 1$) implies $\nu-e_j\in \Lambda$ where $e_j$ is the $j$th unit vector in the standard basis of ${\mathbb R}^J$. 
		\end{enumerate}
		
	\end{definition}
	
	Interpolation on lower sets has been extensively studied (see, e.g., \cite{Chkifa2013,Dyn2014,Sauer2003,Werner1980}). We define the interpolation operator on a lower set $\Lambda$ as 
	\begin{equation} \label{sparse_lower}
		I_\Lambda : = \sum_{\nu \in \Lambda} \Delta_\nu. 
	\end{equation}
	The interpolant $I_\Lambda$ is the unique polynomial in the space 
	\begin{equation*}
		P_\Lambda:=\operatorname{span}\{\by^\nu:\nu\in \Lambda\}
	\end{equation*}
	that agrees with $g$ at the points $\bz_\nu$ for all $\nu\in \Lambda$ (see, e.g., \cite{Chkifa2013}).

	\subsection{Newton-type formula} 
	To compute the interpolant $I_{\Lambda}$ for a lower set $\Lambda$, we express \eqref{sparse_lower} more explicitly. We start with the univariate difference operator $\Delta_k=I_k-I_{k-1}$. For a function $g\in C([-1,1])$, since $\Delta_kg$ is a polynomial of degree $k$ and   $\Delta_kg(z_j)=0$ for $j=0,\dots,k-1$, we have
	\begin{equation*}
		\Delta_kg(z) =c_k(g) (z-z_0)\dots (z-z_{k-1}), 
	\end{equation*} 
	where $c_k(g)$ is a constant that depends on $g$. Thus, in the univariate setting, \eqref{sparse_lower} coincides with the well-known Newton's formula for Lagrange interpolation: 
	\begin{equation}\label{newton_1d}
		I_ng(z) = c_0(g)+c_1(g)(z-z_0)+\dots+c_n(g)(z-z_0)\dots (z-z_{n-1}). 
	\end{equation} 
	Newton's formula allows for efficient computation of the interpolation polynomial when new interpolating points are added sequentially.  
	
	Next, we determine the coefficients in expansion \eqref{newton_1d}. For each $k\ge 0$, let $h_k$ be the polynomial of degree $k$ such that $h_k(z_k)=1$ and $h_k(z_j)=0$ for $j=1,\dots,k-1$, that is, 
	\begin{equation*}
		h_k(z):= \prod_{j=0}^{k-1}\frac{z-z_j}{z_k-z_j},
	\end{equation*}
	with the convention $h_0(z):=1$. Note that $\Delta_k$ is a multiple of $h_k$ and $\Delta_kg(z_k)=g(z_k)-I_{k-1}g(z_k)$. Hence 
	\begin{equation*}
		\Delta_kg(z) = \alpha_k(g)h_k(z),
	\end{equation*}
	where 
	\begin{equation*}
		\alpha_k(g) := g(z_k)-I_{k-1}g(z_k). 
	\end{equation*}
	We can rewrite \eqref{newton_1d} as 
	\begin{equation*}
		I_ng(z) = \sum_{k=0}^{n} \alpha_k(g) h_k(z).  
	\end{equation*}
	In the literature, the weights $\alpha_k$ are referred to as `hierarchical surplus' \cite{Klimke2005}. 
	
	The Newton formula can be generalized to the multivariate setting by using the tensorized hierarchical polynomial 
	\begin{equation*}
		H_\nu = \otimes_{j=1}^{J} h_{\nu_j}. 
	\end{equation*}
	Let $(\Lambda_n)_{n\ge 1}$ be a nested sequence of lower sets, with $\Lambda_n = \{\nu^1,\dots,\nu^n\}$. Similar to the univariate case, for $g\in C([-1,1]^J)$, by tensorization of $H$ and $\Delta$, we have 
	\begin{equation*}
		\Delta_{\nu^n}(g) = \alpha_{\nu^n}(g)H_{\nu^n}, \quad \alpha_{\nu^n}(g):=g(z_{\nu^n})-I_{\Lambda_{n-1}}g(z_{\nu^n}).
	\end{equation*}
	Hence, from \eqref{sparse_lower},
	\begin{equation}\label{Newton_mult}
		I_{\Lambda_n}g = \sum_{k=1}^{n}\alpha_{\nu^k}(g)H_{\nu^k}.
	\end{equation}
	Thus, we can recursively update new weights and new interpolants as follows:
	\begin{equation}
		\begin{split}
			\alpha_{\nu^n}(g)  &=  g(z_{\nu^n}) - \sum_{k=0}^{n-1}\alpha_{\nu^k}(g) H_{\nu^k}(z_{\nu^n}), \\ 
			I_{\Lambda_n}g &=  I_{\Lambda_{n-1}}g + \alpha_{\nu^n}(g) H_{\nu^n}. 
		\end{split}
	\end{equation}
	Our aim is to approximate the finite element approximation ${\mathcal G}^{J,l}$ by its interpolation over a family of lower sets $\Lambda_N$, $N\in{\mathbb N}$ where the cardinality of $\Lambda_N$ is $N$. Ideally, we wish to obtain an explicit error estimate for this approximation as a negative power of $N$. We give in subsection \ref{sec:explicit_rate} an example of a choice of the interpolation points $z_0, z_1, \ldots$ where there exists such a family of lower sets with an explicit rate of convergence.  
	%Theoretically, in particular situations such as the solution of the parametric EiT problem considered in this paper, we can prove the existence of a nested sequence of lower sets $\{\Lambda_n\}_n$ so that the cardinality of $\Lambda_n$ is $n$, and the convergence rate of $I_{\Lambda_n}g$ to $g$ is of a negative power of $n$ (see {\todo refer to the particular parts where we show this, e.g. the appendix}. 
	However, in practice, an explicit constructive algorithm to find such a sequence of nested lower sets with an explicit error estimate in terms of their cardinalities, is generally not available. We present in Subsection \ref{sec:adaptiveint} an adaptive algorithm to construct such a nested sequence of lower sets (\cite{Chkifa2013,Gerstner2003})). This algorithm provides an unambiguous procedure to construct these sets. However, establishing the convergence rate for the sparse adaptive interpolation operator is in general still an open question. 
	
	\subsection{Lower sets with explicit convergence rates}\label{sec:explicit_rate}
	A classical concept in the analysis of interpolation error is that of Lebesgue constant. The Lebesgue constant $\lambda_n$ is defined by the norm of the linear operator $I_n$, defined in \eqref{Lagrange_1d}, from $C[-1,1]$ into itself. In other words,
	\begin{equation*}
		\lambda_n:= \sup_{g\in C[-1,1],g\neq 0} \frac{||I_ng||_\infty}{||g||_\infty}.
	\end{equation*} 
	We assume:
	\begin{assumption}\label{lebesgue_const}
		The univariate sequence $(z_n)_{n\ge 0}$ are chosen so that
		\begin{equation*}
			\lambda_n\le (n+1)^\theta, n\ge 0\;  \text{ for some }\theta\ge 1.
		\end{equation*}
	\end{assumption}
	
	%\begin{remark}\label{remark:Leja}
	%\end{remark}
	\begin{example}
		The Leja sequence on the complex unit disk $\mathcal{D}:=\{\xi\in \mathbb{C}: |\xi|\le 1\}$ is defined by
		\begin{equation*}
			l_n:=\argmax_{\xi\in \mathcal{D}}\prod_{j=1}^{n-1}|\xi-l_j|.
		\end{equation*}
		The projection of $(l_n)_{n\ge 0}$ onto $[-1,1]$ are called the $R$-Leja sequence on $[-1,1]$. For $R$-Leja sequence, it is shown in \cite{Chkifa2015P} that $\lambda_n =O((n+1)^2)$. 
	\end{example}
	\begin{proposition}\label{best_N_existence}
		Under Assumptions \ref{lebesgue_const} and \ref{lp-summability_assumption}, there exists a nested sequence of lower sets $(\Lambda_N)_{N\ge 1}$ such that $\#(\Lambda_N)=N$ and  
		\begin{equation*}
			\sup_{\by\in U_J}|\mathcal{G}^{J,l}(\by)-I_{\Lambda_N}\mathcal{G}^{J,l}(\by)|_\Sigma\le CN^{-s},
		\end{equation*}
		where $C$ is a constant independent of $J,l,N$. 
	\end{proposition}
	We prove this in Appendix \ref{section:error_estimate}.
	Proposition \ref{best_N_existence} establishes the existence of a nested sequence of lower sets.  However, a constructive algorithm to find these sets is not available. In the next section, we present a greedy algorithm to find a sequence of lower sets. Though the convergence rate in terms of the cardinality of the sets is still an open question, the algorithm is unambiguous. 
	\subsection{Sparse adaptive interpolation}\label{sec:adaptiveint}
	Following \cite{Gerstner2003, Chkifa2013}, we describe a greedy algorithm to select the lower set $\Lambda$ adaptively. It is an iterative algorithm, where in each step we add a new index $\nu$ to the current index set $\Lambda$. Motivated by expansion \eqref{sparse_lower}, for each iteration, we aim to choose $\nu\notin \Lambda$ such that $||\Delta_{\nu}f||_\infty$ is largest. In this way, we hope that a large $L^\infty$-error reduction is achieved after each step. 
	
	However, to maintain that the updated index set is always a lower set after each iteration, we limit the search to the neighbour $\mathcal{N}(\Lambda)$ of the current lower set, where $\mathcal{N}(\Lambda)$ is defined as the set of indices $\nu\notin \Lambda$ so that $\Lambda \cup \{\nu\}$ is also a lower set. Equivalently, $\mathcal{N}(\Lambda)$ consists of those $\nu \notin \Lambda$ such that $\nu-e_j\in \Lambda$ for all $j$ such that $\nu_j\ge 1$.  We summarize the adaptive algorithm for interpolation on lower sets in Algorithm \ref{alg:adaptive_lower} (cf. \cite{Chkifa2013, Gerstner2003}).

	\begin{algorithm}
		\caption{Adaptive selection of the index set for interpolation}
		\label{alg:adaptive_lower}
		\begin{algorithmic}[1]
			\Function{Adaptive\_Interp}{$f,\boldsymbol{z},N$}       
			\State$\nu^1 := (0,\dots,0)$
			\State $\Lambda := \{\nu^1\}$
			\For{$n=2,\dots,N$} 
			\State $\nu^n:= \argmax_{\nu \in \mathcal{N}(\Lambda)} ||\Delta_\nu f||_{\infty}$
			\State $\Lambda := \Lambda \cup \{\nu^n\}$
			\EndFor
			\State \Return $\Lambda$
			\EndFunction
		\end{algorithmic}
	\end{algorithm}
	
	In this algorithm, since we add new interpolation nodes sequentially, Newton formula is suitable. For each $n$, let $\Lambda_n:=\{\nu^1,\dots,\nu^n\}$. Recall that 
	\begin{equation*}
		||\Delta_{\nu^n} f||_\infty =  |\alpha_{\nu^n}(f)|.||H_{\nu^n}||_\infty, 
	\end{equation*}
	and 
	\begin{equation*}
		\alpha_{\nu^n}(f)  =  f(z_{\nu^n}) - \sum_{k=0}^{n-1}\alpha_{\nu^k}(f) H_{\nu^k}(z_{\nu^n}).
	\end{equation*}

	\subsection{Approximation of the posterior by Lagrange interpolation of the forward map}
	Let $\Lambda_N\subset \mathbb{N}_0^J$ be a lower set of  cardinality $N$. We approximate the forward map $\mathcal{G}$ by the interpolant of $\mathcal{G}^{J,l}$, that is,  
	\begin{equation}\label{G_Nl_def}
		\mathcal{G}^{J,l,N}:=I_{\Lambda_N}\mathcal{G}^{J,l}. 
	\end{equation}
	Replacing $\mathcal{G}^{J,l}$ by $\mathcal{G}^{J,l,N}$ in \eqref{eq:PhiJl}, we get an approximation $\Phi^{J,l,N}$ of $\Phi^J$, defined by
	\begin{equation}
		\Phi^{J,l,N}(\by;\delta):=\frac{1}{2}|\delta-\mathcal{G}^{J,l,N}(\by)|_{\Sigma}^2. 
		\label{eq:PhiJlN}
	\end{equation}
	We then approximate the posterior $\mu^{J,l,\delta}$ by the measure  $\mu^{J,l,N,\delta}$ defined by 
	\begin{equation}
		\frac{d\mu^{J,l,N,\delta}}{d \mu_0^J}(\by)\propto \exp(-\Phi^{J,l,N}(\by;\delta))
		\label{eq:muJlN}
	\end{equation}
	\begin{comment}
		\begin{proposition}
			\todo{
				Assuming that the sequence of lower sets $\Lambda_N$ satisfies the convergence rate in Proposition \ref{best_N_existence}, we then have
				\[
				d_{Hell}(\mu^\delta,\mu^{J,l,N,\delta})\le c(J^{-s}+2^{-2l}+N^{-s}).
				\]
				And say very briefly how to prove it.
			}
		\end{proposition}
	\end{comment}
	
		\begin{proposition}\label{Hell_uniform_bound}
		Under Assumptions \ref{lp-summability_assumption}, \ref{assumption:sigma,zeta} and \ref{lebesgue_const}, there exists a nested sequence of lower sets $(\Lambda_N)_{N\ge 1}$ such that $\#(\Lambda_N)=N$ and  
		\begin{equation*}
			d_{Hell}(\mu^{J,l,N,\delta},\mu^{J,\delta})\le C(2^{-2l}+N^{-s}),
		\end{equation*}
		where $C$ is a constant independent of $J,l,N$. 
	\end{proposition}
	\begin{proof}
		Combining Propositions \ref{G-Gl} and \ref{best_N_existence}, by the triangle inequality, 
		\begin{equation}\label{G-G_Nl}
			\sup_{\by\in U_J}|\mathcal{G}^{J}(\by)-\mathcal{G}^{J,l,N}(\by)|_\Sigma\le C(2^{-2l}+N^{-s}). 
		\end{equation}
		
		Recall from \eqref{GJ_bounded} that $|\mathcal{G}^J(\by)|$ is uniformly bounded for every $\by\in U_J, J\in \mathbb{N}$. Together with \eqref{G-G_Nl}, we deduce that $|\mathcal{G}^{J,N,l}(\by)|$ is uniformly bounded for every $\by\in U_J$ and $J,l,N\in \mathbb{N}$. Mimicking the proof of Proposition \ref{mu-muJ}, we get the desired estimate. 
	\end{proof} 
	\section{Sparse interpolation Metropolis-Hastings MCMC}\label{sect:MH_MCMC}
	We perform MCMC to sample the approximating posterior measure $\mu^{J,l,N,\delta}$ obtained from sparse Lagrange interpolation. The advantage of using Lagrange interpolation is that the approximations for the solution of the forward equation can be obtained before running the MCMC process, avoiding the expensive process of solving a realization of the forward equation with high accuracy at each sample. 
	For an overview of Metropolis-Hastings methods, we refer to \cite{Roberts2004, Meyn2009, Dashti2017} and the references therein.  
	
	We sample the posterior measure $\mu^{J,l,N,\delta}$ %(which, in the description of the MCMC method below, we denote it simply by $\mu$ for conciseness) 
defined in \eqref{eq:muJlN} in the measurable space $(U_J:=[-1,1]^{J}, \otimes_{j=1}^J{\mathcal B}[-1,1])$ with the prior  measure $\mu_0^J$.  
	%We sample the posterior measure $\mu^J$ of the form
	%\begin{equation*}
	%	\mu(dy)\propto \exp(-\Phi(y)) \mu_0(dy). 
	%\end{equation*}
	%In this study, $Y$ can be $[-1,1]^{\mathbb{N}}$ or $[-1,1]^J$ for some $J\in \mathbb{N}$; $\mu_0, \Phi, \mu$ can be the corresponding prior, misfit, posterior. 
	Starting with a proposal kernel $Q^J({\by},d{\bs})$, Metropolis-Hastings methods construct a Markov chain reversible with respect to the target measure $\mu$. Let 
	\begin{equation*}
		\eta(d\by,d\bs):=\mu^{J,l,N,\delta}(d\by)\otimes Q^J(\by,d\bs), \quad  \eta^\perp (d\by,d\bs):=\mu^{J,l,N,\delta}(d\bs)\otimes Q^J(\bs,d\by). 
	\end{equation*} 
	Provided that  $\eta$ and $\eta^\perp$ are equivalent measures, we define the acceptance probability as
	\begin{equation*}
		\alpha(\by,\bs) := \min\{1, \frac{d\eta^\perp}{d\eta}(\by,\bs)\}. 
	\end{equation*}
	We then construct a Markov Chain $(\by^{(n)})_{n\ge 1}$ as follows
	\begin{itemize}
		\item initialize $\by^{(0)}$; 
		\item for $k=1,2,\ldots$, propose a candidate $\bs^{(k)} \sim Q^J(\by^{(k)},d\bs)$; 
		\item set $\by^{(k+1)}=\bs^{(k)}$ with probability $\alpha(\by^{(k)},\bs^{(k)})$ and $\by^{(k+1)}=\by^{(k)}$ with probability $1-\alpha(\by^{(k)},\bs^{(k)})$. 
	\end{itemize}

	\begin{comment}
			In other words, the transition kernel for the Markov chain $(\by^{(n)})_{n\ge 1}$ is 
		\begin{equation*}
			K(\by,d\bs) := \alpha(\by,\bs)Q(\by,d\bs) + \delta_{\by}(d\bs) \left(1 - \int_U \alpha(\by,\bs)Q(\by,d\bs)\right).
		\end{equation*}
		The Markov Chain $(\by_n)_{n\ge 1}$ is then reversible with respect to $\mu$, that is, 
		\begin{equation*}
			\mu(d\by)K(\by,d\bs)=\mu(d\bs)K(\bs,d\by). 
		\end{equation*} 
	\end{comment}

	%Next, we discuss the choice of the proposal kernel $Q$. We are interested in algorithms such that the acceptance probability is well-defined in the case where there are infinite parameters, i.e., $\eta$ and $\eta^\perp$ are equivalent measures in the case $Y=[-1,1]^{\mathbb{N}}$. In practice, this condition means that when we run MCMC on the truncated parameter space $[-1,1]^J$, the acceptance probability will not decrease to $0$ when we increase the number $J$ of parameters. This is in contrast with the behaviour of the standard Random Walk Metropolis algorithm. 
	
	A sufficient condition to guarantee that $\eta$ and $\eta^\perp$ are equivalent measures is that the proposal kernel $Q^J$ is reversible with respect to the reference measure $\mu_0^J$ (Theorem 22 in \cite{Dashti2017}) 
	\begin{equation}\label{proposal_reversible}
		\mu_0^J(d\by)Q^J(\by,d\bs)=\mu_0^J(d\bs)Q^J(\bs,d\by). 
	\end{equation}
	Moreover, when \eqref{proposal_reversible} holds, the acceptance probability is
	\begin{equation}\label{eq:alphaJlN}
		\alpha^{J,l,N}(\by,\bs) = \min\{1, \exp(\Phi^{J,l,N,\delta}(\by)-\Phi^{J,l,N,\delta}(\bs)) \}. 
	\end{equation}
	 We approximate $E^{\mu^\delta}[\sigma]$ by 
		\begin{equation*}
			E_M^{\mu^{J,l,N,\delta}}[\sigma^J]:= \frac{1}{M}\sum_{i=1}^{M}\sigma^J(\by^{(i)}), 
		\end{equation*}
	where $(\by^{(i)})_{i=1}^{\infty}$ is the Markov chain obtained from the MCMC process.
	
	We give some examples of  proposal distributions that satisfy \eqref{proposal_reversible}. 
	The simplest proposal kernel is that of the independence sampler  method %\cite{Hoang2013}: 
	\begin{equation*}
		Q_{IS}^J(\by,d\bs):=\mu_0^J(d\bs). 
	\end{equation*}
	In practice, the IS algorithm often leads to slow mixing because the proposal does not take advantage of the previous states of the chain. Reflection random walk Metropolis (RRWM) \cite{Vollmer2015} is a local Metropolis-Hastings method which is a modification of the standard Random Walk Metropolis. On the one-dimensional probability space $[-1,1]$, the proposal $q_{RRWM}(y,ds)$ is defined as the law of the random variable $R(y+\beta \xi)$, where $\beta\in (0,1)$, $\xi\sim Uniform[-1,1]$, and $R$ is the reflection via the endpoints $\{-1,1\}$: 
	\begin{equation*}
		R(y):= \begin{cases}
			-2-y\quad \text{if } y<-1\\
			y \quad \text{if } y\in [-1,1]\\
			2-y \quad \text{if } y>1. 
		\end{cases}
	\end{equation*} 
	For $U_J=[-1,1]^J$, the proposal kernel is defined as the tensorization of the one-dimensional proposal kernel 
	\begin{equation*}
		Q_{RRWM}^{J}(\by,d\bs):= \otimes_{j= 1}^{J} q_{RRWM}(y_j,ds_j)
	\end{equation*}
	for $\by=(y_1,\ldots,y_J)\in U_J$ and $\bs=(s_1,\ldots,s_J)\in U_J$.
	Another method is described in \cite{Chen2019}, where the state space $[-1,1]^J$ is mapped into the new state space $\mathbb{R}^J$ by the inverse of the standard normal cumulative distribution function and the preconditioned Crank Nicolson MCMC method \cite{Cotter2013} is applied to $\mathbb{R}^J$.

	 %{\todo Define $E_M^{\mu^{J,l,M,\delta}}$ here. What do you mean?}

	 We have the following result for the independence sampler and the reflection random walk Metropolis. By $\mathcal{E}^{{\mu_0^J},J,l,N}$, we denote the expectation taken with respect to the joint distribution of the Markov chain $(\by^{(i)})_{i=1}^{\infty}$ with the acceptance probability $\alpha^{J,l,N}$ in \eqref{eq:alphaJlN}, and the initial sample being distributed according to $\mu_0^J$.  We prove this in Appendix \ref{section:error_estimate}. 
	%We use the following quantity to measure the error in the approximation of the posterior mean $E^{\mu^\delta}[\sigma(y)]$:
	\begin{theorem} \label{error_total_J}
		%\todo{make necessary assumptions here} 
		Under Assumptions \ref{lp-summability_assumption}, \ref{assumption:sigma,zeta}, \ref{lebesgue_const}, there exists a nested sequence of lower sets $(\Lambda_N)_{N\ge 1}$ such that $\#(\Lambda_N)=N$ and  
		\begin{equation*}
			\varepsilon_{J,l,N,M}:=\left(\mathcal{E}^{\mu_0^J,J,l,N}\left[\left\Vert E^{\mu^{\delta}}[\sigma]- E_M^{\mu^{J,l,N,\delta}}[\sigma^J]\right\Vert_{L^\infty(D)}^2\right]\right)^{1/2}\le c(J^{-s}+2^{-2l}+M^{-1/2}+N^{-s})
		\end{equation*}
	\end{theorem}

	To get the rate of convergence in this theorem, we assume the interpolation convergence rate $N^{-s}$ for the sequence of lower sets. Practically, a constructive algorithm to identify such a sequence of lower sets with an explicit interpolation convergence rate has not been available. 
	In all the numerical experiments in this paper, we will use Algorithm \ref{alg:SAI-MCMC}, where the index set $\Lambda$ for interpolation is chosen adaptively by Algorithm \ref{alg:adaptive_lower}. From Algorithm \ref{alg:SAI-MCMC}, we see that for each MCMC iteration, we evaluate a polynomial instead of solving a forward PDE. 
	
	\begin{algorithm}
		\caption{Sparse adaptive interpolation-MCMC}
		\label{alg:SAI-MCMC}
		\begin{algorithmic}[1]
			\Require $\boldsymbol{z},N,\mathcal{G}^{J,l},\delta,u^{(1)},M,n_B$
			\State $[\Lambda,\alpha] := \text{ADAPTIVE\_INTERP}(\mathcal{G}^{J,l},\boldsymbol{z},N)$
			\Comment{Offline}
			\For{$i=1,\dots,M-1$}
			\State Propose $v^{(i)}\sim Q(u^{(i)},dv)$
			\State Evaluate $\mathcal{G}^{J,l,N}(v^{(i)})=\sum_{n=1}^{N}\alpha_n H_{\nu^n}(v^{(i)})$
			\State Calculate $\alpha^{J,l,N}(u^{(i)},v^{(i)})$ according to \eqref{eq:alphaJlN}
			\State Draw $rand\sim Uniform[0,1]$.
			\If{$\alpha^{J,l,N}(u^{(i)},v^{(i)})<rand$} 
			\State 	$u^{(i+1)}:=v^{(i)}$
			\Else 
			\State 	$u^{(i+1)}:=u^{(i)}$
			\EndIf
			\EndFor
			%		\State Burn-in: discard $u^{(1)},\dots,u^{(n_B)}$
			%		\State Compute the Monte Carlo approximation 
			%		\begin{equation*}
				%			y_{mean}:=\frac{1}{M-n_B}\sum_{i=n_B+1}^{M}u^{(i)} 
				%		\end{equation*} 
			%		\State Compute the approximation for the posterior mean of $\sigma$
			%		\begin{equation*}
				%			E^{\mu^{N,l,\delta}}[\sigma]\approx \sigma(\cdot,y_{mean}). 
				%		\end{equation*}
		\end{algorithmic}
	\end{algorithm}
	
	The sparse adaptive interpolation MCMC  algorithm consists of two stages. In the offline stage (pre-processing phase), we find the polynomial $\mathcal{G}^{J,l,N}$ that approximates the forward map $\mathcal{G}$. In the online stage (post-processing phase), we run MCMC iterations using evaluations of $\mathcal{G}^{J,l,N}$. 
	
	As in other surrogate methods for Bayesian inverse problems, an advantage of sparse interpolation-MCMC is that it only requires polynomial evaluations in the online phase, avoiding solving the forward PDE repeatedly. Furthermore, we may use the same polynomial from the offline stage whenever new data arrive.

	\section{Numerical examples} \label{sect:numeric}
	
	We investigate the computation time and reconstructions using sparse adaptive interpolation MCMC. Although in the previous sections, we only consider the case where the conductivity is a linear expansion of uniformly bounded random variables, the method works equally in the log-uniform case where the log conductivity is a linear expansion of uniformly distributed random variables in a compact interval. We consider a  wavelet prior and a trigonometric prior in the experiments below. Throughout, we carry computations on a computer with 32 GB RAM and 3.70 GHz processor, using MATLAB 2020b. 
	
	We choose the domain $D$ as the unit square $(0,1)^2$.  We use $K=16$ or $K=64$ electrodes, with equal length, covering half of the boundary of $D$. The electrodes' positions with $K=16$ are depicted in \cref{fig:electrode_position}. We choose the contact admittance $\zeta$ at the electrodes as the hat function  as in \cite{Hyvonen2017}).%with height $1000$. 
We use $K-1$ current patterns 
	\begin{equation*}
		I^{(k)}=e_1-e_{k+1}, \quad k=1,\dots,K-1.
	\end{equation*} 
	
	\begin{figure}
		\centering
		\begin{tikzpicture}
			\draw (0,0) rectangle (4,4);
			\draw[step=0.25cm,gray,very thin] (0,0) grid (4,4);
			\draw[thick,blue] (0.25,0) -- (0.75,0); 
			\draw[thick,blue] (1.25,0) -- (1.75,0);
			\draw[thick,blue] (2.25,0) -- (2.75,0);
			\draw[thick,blue] (3.25,0) -- (3.75,0);
			\draw[thick,blue] (0,0.25) -- (0,0.75);
			\draw[thick,blue] (0,1.25) -- (0,1.75);
			\draw[thick,blue] (0,2.25) -- (0,2.75);
			\draw[thick,blue] (0,3.25) -- (0,3.75);
			\draw[thick,blue] (4,0.25) -- (4,0.75);
			\draw[thick,blue] (4,1.25) -- (4,1.75);
			\draw[thick,blue] (4,2.25) -- (4,2.75);
			\draw[thick,blue] (4,3.25) -- (4,3.75);
			\draw[thick,blue] (0.25,4) -- (0.75,4); 
			\draw[thick,blue] (1.25,4) -- (1.75,4); 
			\draw[thick,blue] (2.25,4) -- (2.75,4); 
			\draw[thick,blue] (3.25,4) -- (3.75,4); 
		\end{tikzpicture}
		\caption{Positions of the 16 electrodes, in blue}
		\label{fig:electrode_position}
	\end{figure}
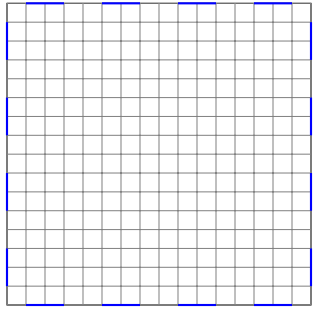

	The numerical solutions of the forward boundary value problem \eqref{SCEM} at the interpolation nodes are computed by FEs with piecewise linear basis functions. %To calculate the stiffness matrix, we adapt the function \verb|bld_master| in the MATLAB package \verb|EIDORS| \cite{Vauhkonen2001}, which employs a formula in \cite{Vavasis1996}. 
	We use mesh size $2^{-6}$ for both data generation and reconstruction. 
	%, with the possibility of inverse crime. 
	%We intend to do so to ignore the error arising from FE approximation, as our focus in this paper is interpolation. We did experiments with mesh size $2^{-8}$ for data generation, $2^{-6}$ for reconstruction. We observed no significant difference in the reconstructions made by this choice and that made by choosing the same mesh size $2^{-6}$ for both purposes.  
	To generate data, we solve the forward PDE by FE using the reference conductivity to get the noise-free voltage vectors $V^{(1)},\dots,V^{(K-1)}$. The data are obtained by adding to the voltage vectors Gaussian noises with a relatively small standard deviation 
	\begin{equation*}
		sd_{noise} = 10^{-4}(V_{max}-V_{min}), 
	\end{equation*}
	where $V_{max}=\max_{j,k} V_j^{(k)}, V_{min}=\min_{j,k} V_j^{(k)}$.

	For interpolation, we choose the univariate sequence $(z_k)$  as the R-Leja sequence on $[-1,1]$, as in subsection \ref{sec:explicit_rate}. The complex Leja points are computed by an exact formula in \cite{Reichel1990}. That is, for an integer $j$ with binary representation 
	\begin{equation*}
		j= \sum_{k=0}^\infty j_k2^k,\quad j_k\in \{0,1\},
	\end{equation*}
	the Leja points $\{z_j^{\mathbb{C}}\}$ on the complex unit disk are: 
	\begin{equation*}
		z_j^{\mathbb{C}}:=\exp(i\pi \sum_{k=0}^{\infty}j_k2^{-k}).
	\end{equation*} 
	The R-Leja points are obtained by projecting these complex Leja points into the real axis.

	\subsection{Parametrizations}
	\subsubsection{Wavelet prior}
	We begin with the  two dimensional Haar wavelet expansion on the domain $[0,1]^2$. In {one dimension}, Haar wavelets on $[0,1]$ are formed from the scaling function $\varphi:=1_{[0,1]}$ and the mother wavelet $\psi:=1_{[0,1/2)}-1_{[1/2,1)}$. In two dimensions, there are three mother wavelets $\psi^1:=\varphi\otimes \psi$, $\psi^2:=\psi \otimes \varphi$, $\psi^3:=\psi\otimes \psi$, as shown in \cref{fig:mother_wavelets}. For $i=1,2,3$, $l\in \mathbb{N}$, $k=(k_1,k_2)\in \{0,\dots,2^l-1\}^2$, let 
	\begin{equation*}
		\psi_{l,k}^i(x):=\psi^i(2^lx-k). 
	\end{equation*}
	The functions $\psi_{l,k}^i$ together with $1_D$ form an orthogonal basis for $L^2([0,1]^2)$. Two-dimensional Haar wavelets are used in \cite{Chkifa2012, Chkifa2013} for parametric PDEs, and in \cite{Wacker2019} for Bayesian inverse problems.  
	
	\begin{figure}
		\centering
		\includegraphics[scale=0.55]{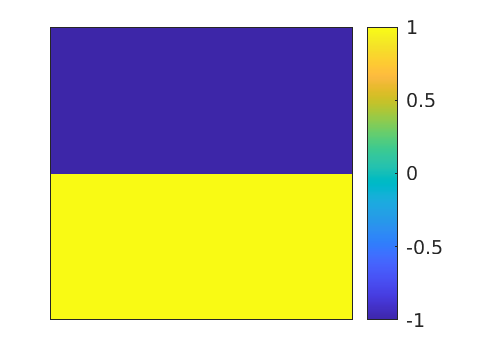} 
		\includegraphics[scale=0.55]{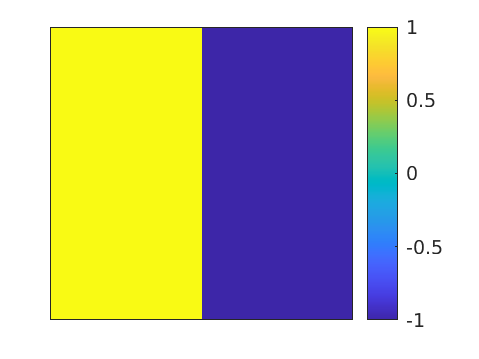} 
		\includegraphics[scale=0.55]{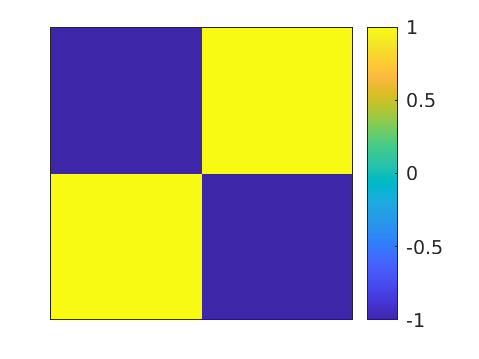} 
		\caption{Three mother wavelets in 2d. From left to right: $\psi^1, \psi^2, \psi^3$.}
		\label{fig:mother_wavelets}
	\end{figure}
	
	Following  Chkifa et al. \cite{Chkifa2013}, we assume that the conductivity $\sigma$ admits the  parametrization
	\begin{equation}\label{wavelet_para}
		\sigma(x,y) := \bar{\sigma} + c\sum_{l=0}^\infty 2^{-\gamma l} \sum_{i=1}^3 \sum_{k\in \{0,\dots,2^l-1\}^2} y_{l,k,i}\psi_{l,k}^i(x),
	\end{equation} 
	where $c,\gamma$ are positive constants. 
	The upper and lower bounds for $\sigma$ are 
	$\bar{\sigma}\pm c\sum_{l=0}^{\infty}2^{-\gamma l}$.
	We choose $\bar{\sigma}:=1.1$, $c:=0.3(1-2^{-\gamma})$ as in \cite{Chkifa2013}, so that the conductivity $\sigma$ always belongs to the bounded interval $[0.2,2.0]$. 
	We convert the expansion \eqref{wavelet_para} into the affine parametrization \eqref{affine_para} by denoting  
	\begin{equation*}
		\psi_j:= c2^{-\gamma l}\psi_{l,k}^i, y_j:= y_{l,k,i}, 
	\end{equation*}
	where $j=2^{2l}+3(2^l k_1+k_2)+i-1$. It follows that $||\psi_j||_{L^\infty(D)}=O(j^{-\gamma/2})$. Hence, if $\gamma>2$, the decay rate in Assumption \ref{lp-summability_assumption} is satisfied with $s=\gamma/2 -1$.

	\subsubsection{Trigonometric prior}
	We consider a log-affine parametrization $\sigma(x) = e^{u(x)}$, where $u(x)$ is of the form
	\begin{equation}\label{Matern_para}
		u(x_1,x_2) = \sum_{j_1,j_2\ge 0} y_{(j_1,j_2)} \frac{\eta}{(\tau^2+a^2\pi^2(j_1^2+j_2^2))^{\kappa}} \cos\left(\pi j_1x_1\right)\cos\left(\pi j_2 x_2\right).
	\end{equation} 
	Here we choose $\eta=10^{-1}$, $\tau=3\cdot10^{-1}$, $a=10^{-1.5}$ and $\kappa=5$; the random variables $y_{(j_1,j_2)}$ follow the uniform distribution in $[-1,1]$. This mimics the Kahunen-Loeve expansion of a Gaussian random variable of the Whittle-Matern type as considered in \cite{Dunlop2016} for the EIT problem with a Gaussian prior, except that in our case the random variables $y_{j_1,j_2}$ are uniformly distributed in $[-1,1]$. In this case, the corresponding Gaussian covariance
	is $\eta^2(\tau^2-a^2\Delta)^{-2\kappa}$, where 
	%If $y_{(j_1,j_2)}$ were independent $N(0,1)$ random variables, $u\sim N(0,C)$, where  
	%\begin{equation}\label{Matern_cov}
	%	C = \eta^2(\tau^2-\Delta)^{-\gamma'},   
	%\end{equation}
	$\Delta$ is the Laplace operator on $D=[0,1]^2$ with homogeneous Neumann boundary condition. 
	%The covariance of the form \eqref{Matern_cov} has been employed in the Bayesian inversion of EIT in \cite{Dunlop2016}. It is modified from the Matern covariance considered in \cite{Roininen2014}. 
	
	We cast \eqref{Matern_para} into the parametrization \eqref{log_affine_para} by denoting 
	\begin{equation*}
		\psi_j:= \frac{\eta}{(\tau^2+a^2\pi^2(j_1^2+j_2^2))^{\kappa}} \cos\left(\pi j_1x_1\right)\cos\left(\pi j_2 x_2\right), y_{j}:=y_{(j_1,j_2)}, 
	\end{equation*}
	where $j = 0.5(j_1+j_2)(j_1+j_2+1)+j_2+1$. It follows that $||\psi_j||_{L^\infty(D)}=O(j^{-\kappa})$. 
	
	\subsection{Computational time of the sparse adaptive interpolation MCMC}
	The complexity of the adaptive interpolation algorithm has not been rigorously justified. We  study numerically
%	demonstrate in an example how 
	the complexity of Algorithm \ref{alg:SAI-MCMC} with respect to the number $N$ of interpolating nodes,  i.e. the cardinality of the lower set $\Lambda_N$. 
	We measure the CPU time taken to perform the offline and online stages of sparse adaptive interpolation MCMC for various values of $J$ and $N$.  In the online stage, we fix $M=10^4$. %We do not include details about the data as well as the specifications of the two priors as they are not relevant.    
	
	The results are shown in \cref{fig:runningtime_wavelet} and \cref{fig:runningtime_Matern}. We observe that for fixed values of $J$, both the offline and online running times  grow linearly in $N$. In addition, the slopes of the best-fit straight lines only increase moderately as $J$ increases. 
	
	\begin{figure}
		\centering
		
		\includegraphics[scale=0.45]{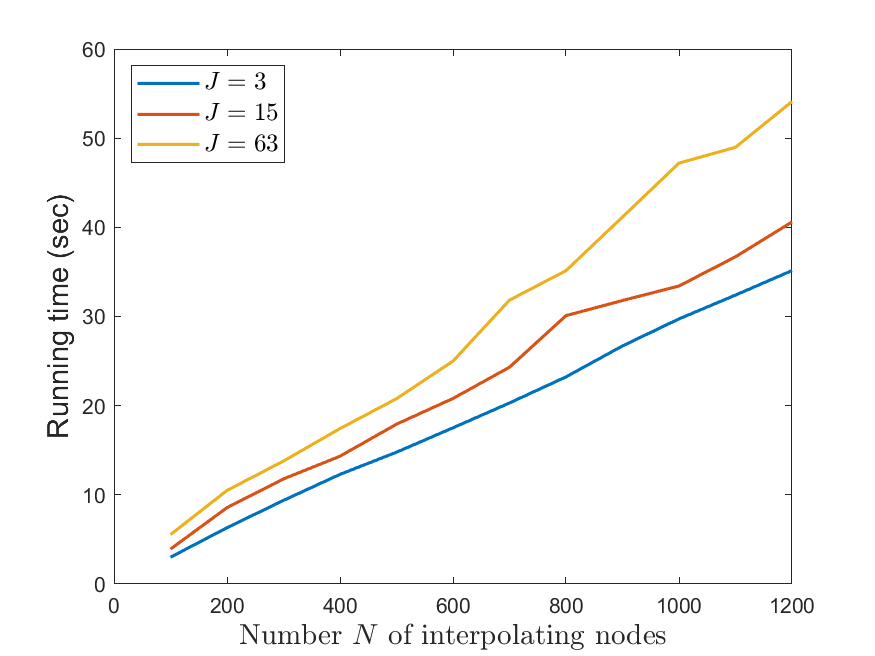} 
		\includegraphics[scale=0.45]{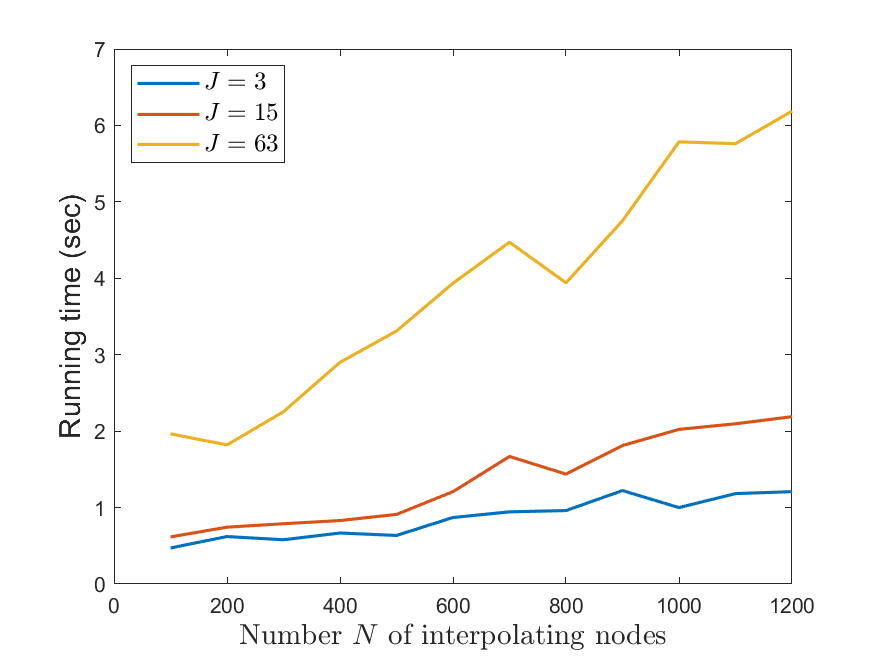} 
		\caption{CPU time of the offline and online stages of sparse adaptive interpolation-MCMC using wavelet prior. Left: offline, right: online.}
		\label{fig:runningtime_wavelet}
	\end{figure} 
	
	\begin{figure}
		\centering
		
		\includegraphics[scale=0.45]{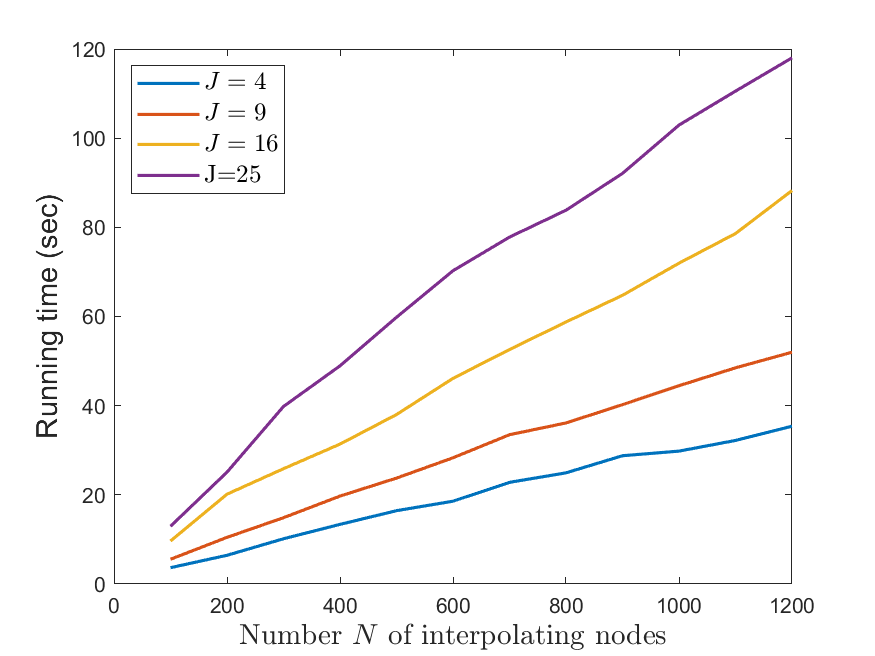} 
		\includegraphics[scale=0.45]{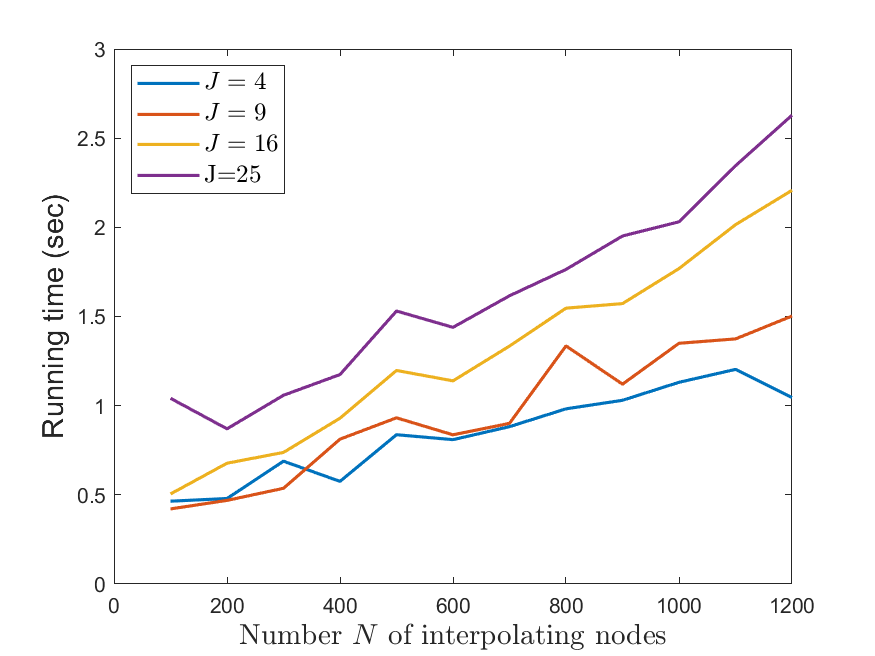} 
		\caption{CPU time of the offline and online stages of  sparse adaptive interpolation-MCMC using trigonometric prior. Left: offline, right: online.}
		\label{fig:runningtime_Matern}
	\end{figure}

	\subsection{Reconstruction}
	
	%For Bayesian inversion, we use the uniform distribution as the prior.  
	The initial state of the MCMC process is chosen randomly according to the uniform distribution. We use sparse adaptive interpolation with $N=5000$ interpolating nodes to approximate the forward map. We choose reflection random walk Metropolis (RRWM) as the MCMC sampling method, discarding the first $20\%$ samples as burn-in. The step size of RRWM is chosen to be $\beta:=0.001$. The acceptance probability is about 10\%-30\% in all experiments considered below. For  the wavelet prior, to have a smoother visualization of the conductivity, we use bilinear interpolation to display the reconstructions. 
	
	First, we consider reconstructions using wavelet prior. We  choose $\gamma=3$ in \eqref{wavelet_para}. 
	The data are generated using $(y_1,y_2,y_3)=(0.025,0.025,-0.025)$, $(y_{10},y_{11},y_{12})=(-0.8,-0.8,-0.8)$ and $y_j=0$ for other $j$'s. The true conductivity with this choice of $\by$ is displayed in the upper left of \cref{fig:16pix}, with a square inclusion of lower conductivity.  
	
	To truncate expansion \eqref{wavelet_para}, we replace it with a finite summation where $l$ goes from $0$ to $L$ for some $L\ge 0$, namely 
	\begin{equation}\label{wavelet_truncated}
			\sigma^J(x,y) := \bar{\sigma} + c\sum_{l=0}^{L} 2^{-\gamma l} \sum_{i=1}^3 \sum_{k\in \{0,\dots,2^l-1\}^2} y_{l,k,i}\psi_{l,k}^i(x),
	\end{equation} 
	where $J=4^{L+1}-1$ is the total number of parameters. The resulting truncated approximation $\sigma^J$ is a piecewise constant function, which can be visualized by  $2^{L+1}\times 2^{L+1}$ pixels.

	In our first experiment, we use $L=1$ in \eqref{wavelet_truncated} for reconstruction, corresponding to $J=15$ parameters.  We use sparse adaptive interpolation MCMC with $M=10^6$ samples. The total time for adaptive selection of the index set $\Lambda$ and running MCMC is about 43 minutes. We then compare the reconstruction by sparse adaptive interpolation MCMC with the two reconstructions by plain MCMC: one with $10^5$ samples that takes 52 minutes to run, and one with $10^6$ samples that takes 10.5 hours.  \cref{fig:16pix} demonstrates that the quality of reconstruction by sparse adaptive interpolation MCMC is comparable to that by plain MCMC with the same number of MCMC samples, with better contrast between the values of the conductivity inside the inclusion, and in the background. Sparse adaptive interpolation MCMC takes far less time than plain MCMC with the same number of samples, as we do not need to solve the forward problem with high accuracy at each sample.  Comparing the recovered conductivity by sparse adaptive interpolation MCMC with $10^6$ samples to that obtained from the plain MCMC with $10^5$ samples, with a comparable running time, we find that the sparse adaptive interpolation MCMC provides a significantly better recovery, where the area of lower conductivity is smaller and is confined to the area of the original inclusion, and the background conductivity is closer to the original.
	% Moreover, given the same amount of time, SAI-MCMC seems to be better at detecting the inclusion's position. 
	If we use a finer FE mesh for reconstruction, we would expect more computational benefits from sparse adaptive interpolation MCMC.

	\begin{figure}
		\centering
		\begin{subfigure}[b]{0.475\textwidth}
			\centering
			\includegraphics[width=\textwidth]{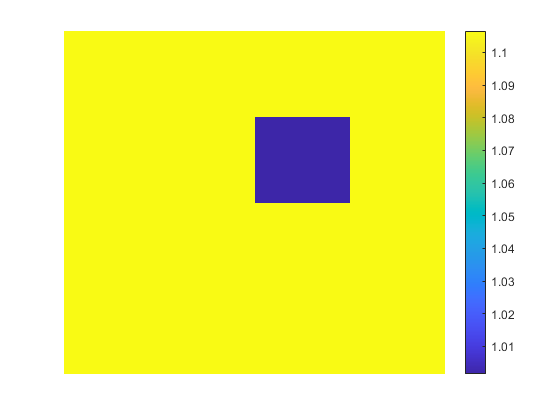}
		\end{subfigure}
		\hfill
		\begin{subfigure}[b]{0.475\textwidth}  
			\centering 
			\includegraphics[width=\textwidth]{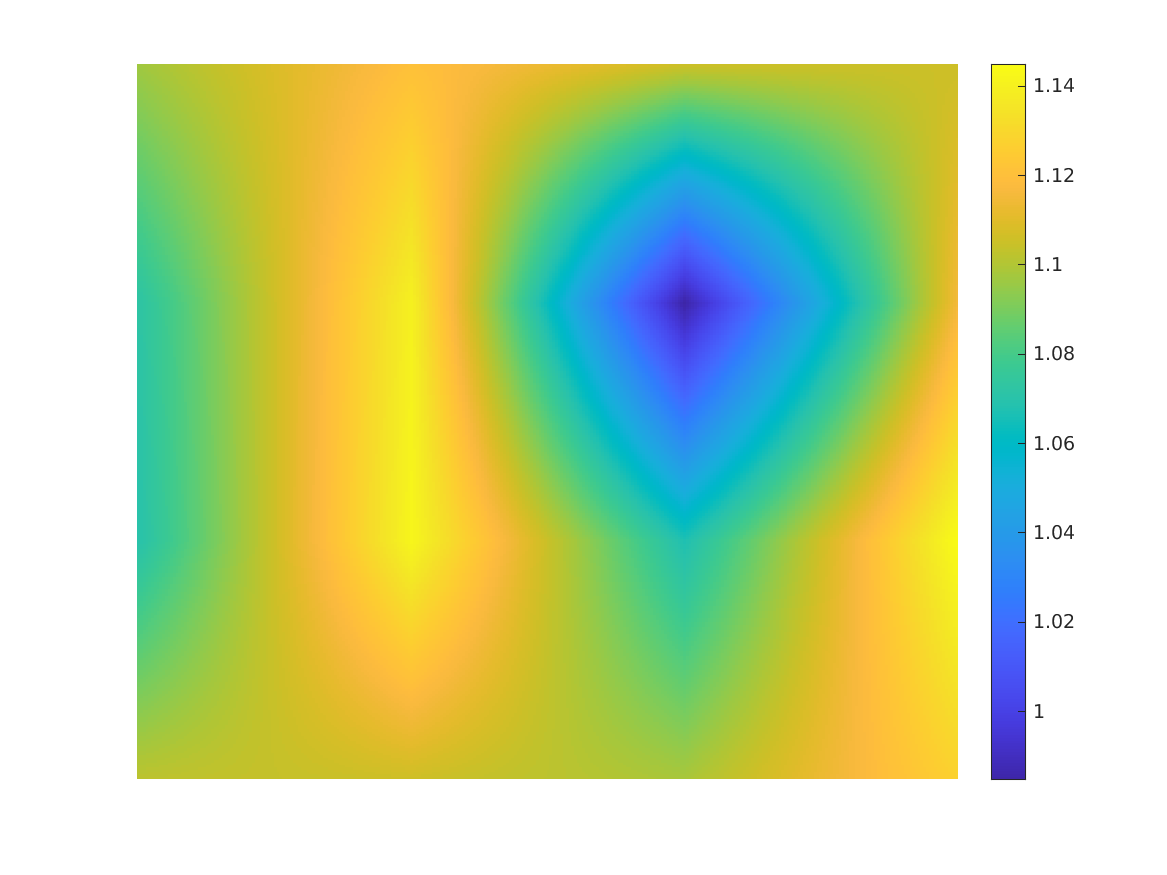}
		\end{subfigure}
		\vskip\baselineskip
		\begin{subfigure}[b]{0.475\textwidth}   
			\centering 
			\includegraphics[width=\textwidth]{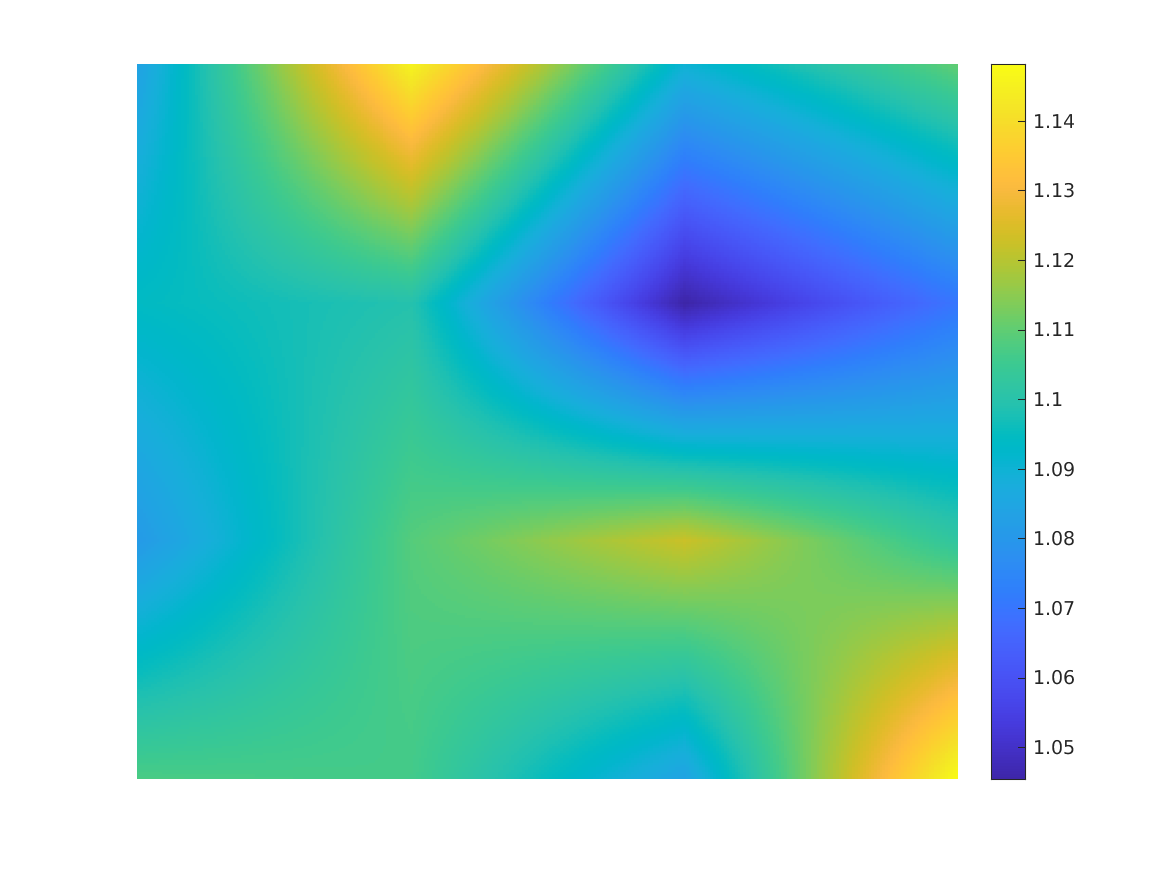}
		\end{subfigure}
		\hfill
		\begin{subfigure}[b]{0.475\textwidth}   
			\centering 
			\includegraphics[width=\textwidth]{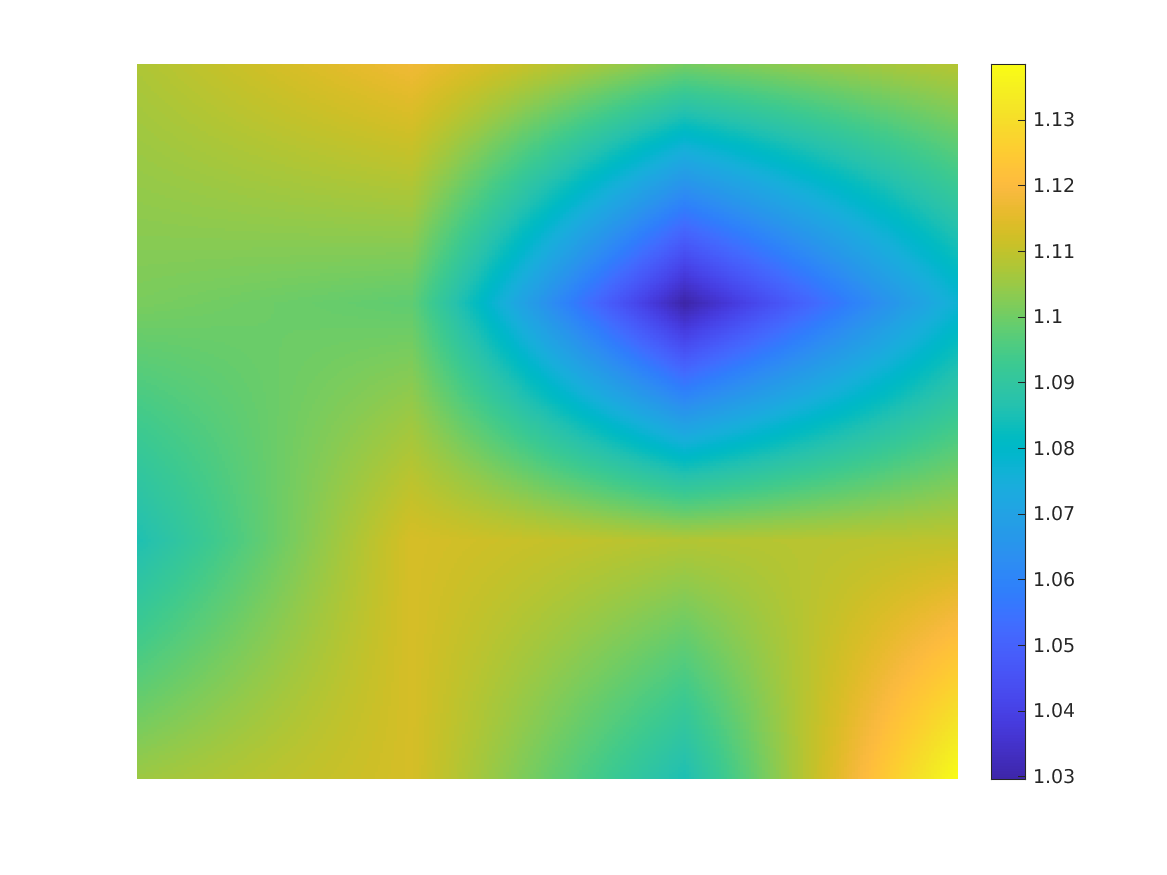}
		\end{subfigure}
		\caption{Reconstruction with wavelet prior using 16 pixels. Upper left: truth. Upper right: reconstruction by sparse adaptive interpolation MCMC with $10^6$ samples and $5000$ interpolating nodes. Lower left: reconstruction by plain MCMC with $10^5$ samples. Lower right: reconstruction by plain MCMC with $10^6$ samples.} 
		\label{fig:16pix}
	\end{figure}

Next, we compare the recovery for the conductivity produced by sparse adaptive interpolation MCMC, where the set of interpolating points $\Lambda$ is chosen by the adaptive algorithm, to that obtained by interpolation MCMC where the interpolating points are chosen isotropically, using a simple bound for the polynomial degree.

{We consider a surrogate using interpolation with the following isotropic index set 
	\begin{equation*}
		Iso(15,5):=\{\nu \in \mathbb{N}_0^{15}: \nu_1+\dots+\nu_{15}\le 5 \}.
	\end{equation*}
The index set $Iso(15,5)$ corresponds to the space of polynomials of 15 variables as in the numerical example above, but with degree at most 5. The cardinality of $Iso(15,5)$ is $\binom{15+5}{5}=15504$. The reconstruction by interpolation MCMC using $Iso(15,5)$ with the same number of MCMC samples, which is $10^6$, as in the previous numerical example, is shown in \cref{fig:iso_15k}. Although the cardinality of $Iso(15,5)$ is significantly larger than that of $\Lambda$, i.e. here we have far more interpolating points, the reconstruction using $Iso(15,5)$ fails to capture the inclusion in the conductivity field. We provide a heuristic explanation as follows. Due to the decay rate of coefficients in the parametrization \eqref{wavelet_truncated}, the components $y_1, y_2, y_3$ contribute significantly more to $\sigma^J(x,y)$ than other components $y_4,\dots, y_{15}$. Table \ref{table_index_set} shows the maximum value and the average for each $\nu_j$ among the 5000 indices in the set $\Lambda$ constructed adaptively in the previous numerical example.  We observe from Table \ref{table_index_set} that most elements in the index set $\Lambda$ are placed in the directions corresponding to $y_1,y_2,y_3$. 
	On the other hand, the average of each component $\nu_j$ of $\nu\in Iso(15,5)$ is $\frac{5}{15+1}=0.3125$. The isotropic approach treats all the parameters equally.
	As a result, the interpolant using $Iso(15,5)$ may not provide a good approximation of the forward map as the one using the anisotropic set $\Lambda$ constructed by the adaptive algorithm. This example shows clearly the advantage of the adaptive approach to construct the set of interpolating points over the isotropic approach. 
	}

	\begin{table}
		\begin{tabular}{l|lllllllllllllll}
			& $\nu_1$ & $\nu_2$ & $\nu_3$ & $\nu_4$ & $\nu_5$ & $\nu_6$ &  $\nu_7$& $\nu_8$ & $\nu_9$ & $\nu_{10}$ & $\nu_{11}$ & $\nu_{12}$ & $\nu_{13}$ & $\nu_{14}$ & $\nu_{15}$ \\ \hline
			Max     &  8       & 8 & 8& 3 & 3 & 3 & 3 & 3 & 3 & 3 & 3 & 3 & 3 & 3 & 3 \\ \hline
			Average &  1.44  &1.36   & 1.47 & 0.17 & 0.17 & 0.15 & 0.12 & 0.09 & 0.09 &0.20  &0.17  &0.16  &0.18  &0.16  & 0.16
		\end{tabular}
		\caption{The maximum and average of components of indices $\nu=(\nu_1,\dots,\nu_{15})$ in the index set $\Lambda$, constructed by sparse adaptive interpolation, with $J=15$ in \eqref{wavelet_truncated}.}
		\label{table_index_set}
	\end{table}
	
	\begin{figure}
		\centering
		
		\includegraphics[scale=0.475]{truth_wavelet_pixx} 
		\includegraphics[scale=0.35]{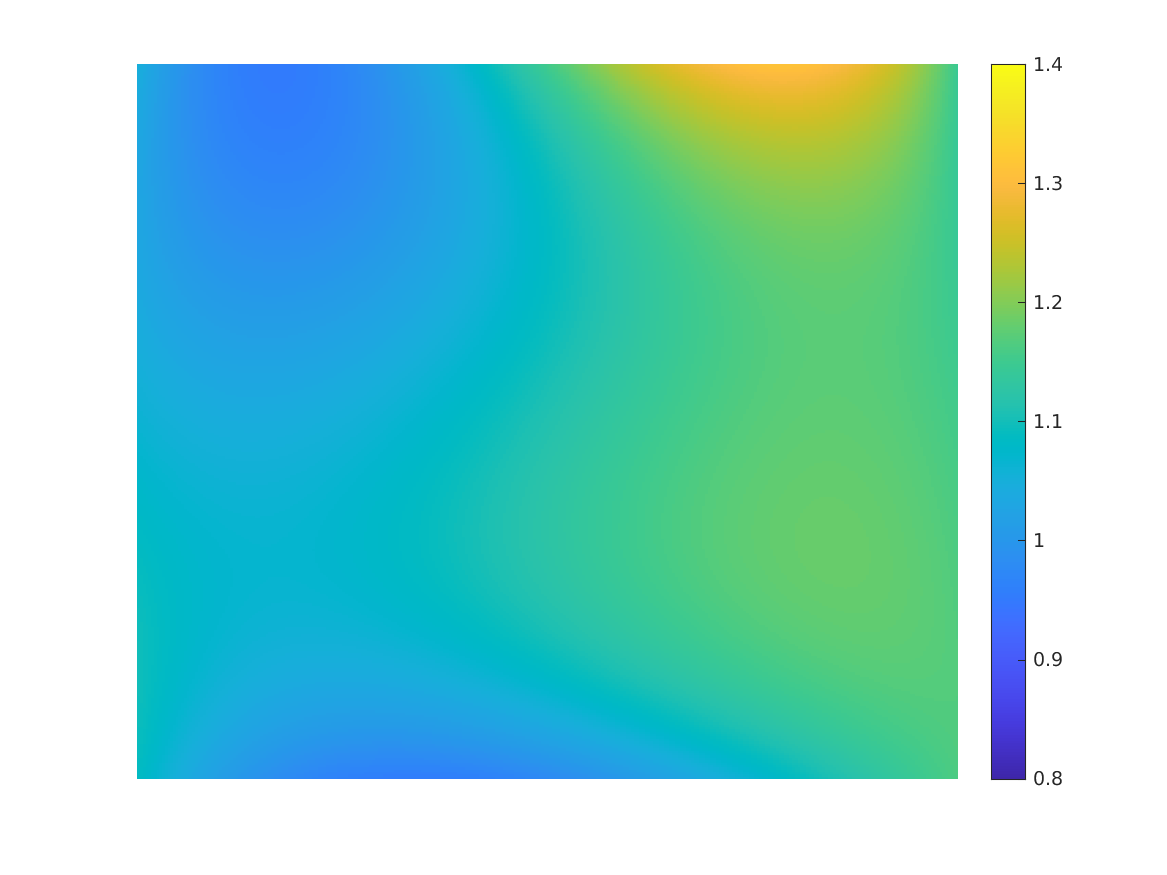} 
		\caption{Reconstructions by sparse interpolation MCMC with the isotropic set $Iso(15,5)$ of cardinality 15504. Left: truth. Right: reconstruction using $Iso(15,5)$.}
		\label{fig:iso_15k}
	\end{figure} 
	
	Next, we use the same reference conductivity for data generation and $L=2$ for reconstruction, corresponding to $J=63$ parameters. We use $M=10^6$ samples. Two constructions with $K=16$ and $K=64$ electrodes are displayed in \cref{fig:64pix}, suggesting that increasing the number of parameters and the data size may enhance reconstruction quality: the panel on the right for the case of 64 electrodes shows better contrast in the values of the conductivity inside and outside of the inclusion.  
	
	\begin{figure}
		\centering
		
		\includegraphics[scale=0.35]{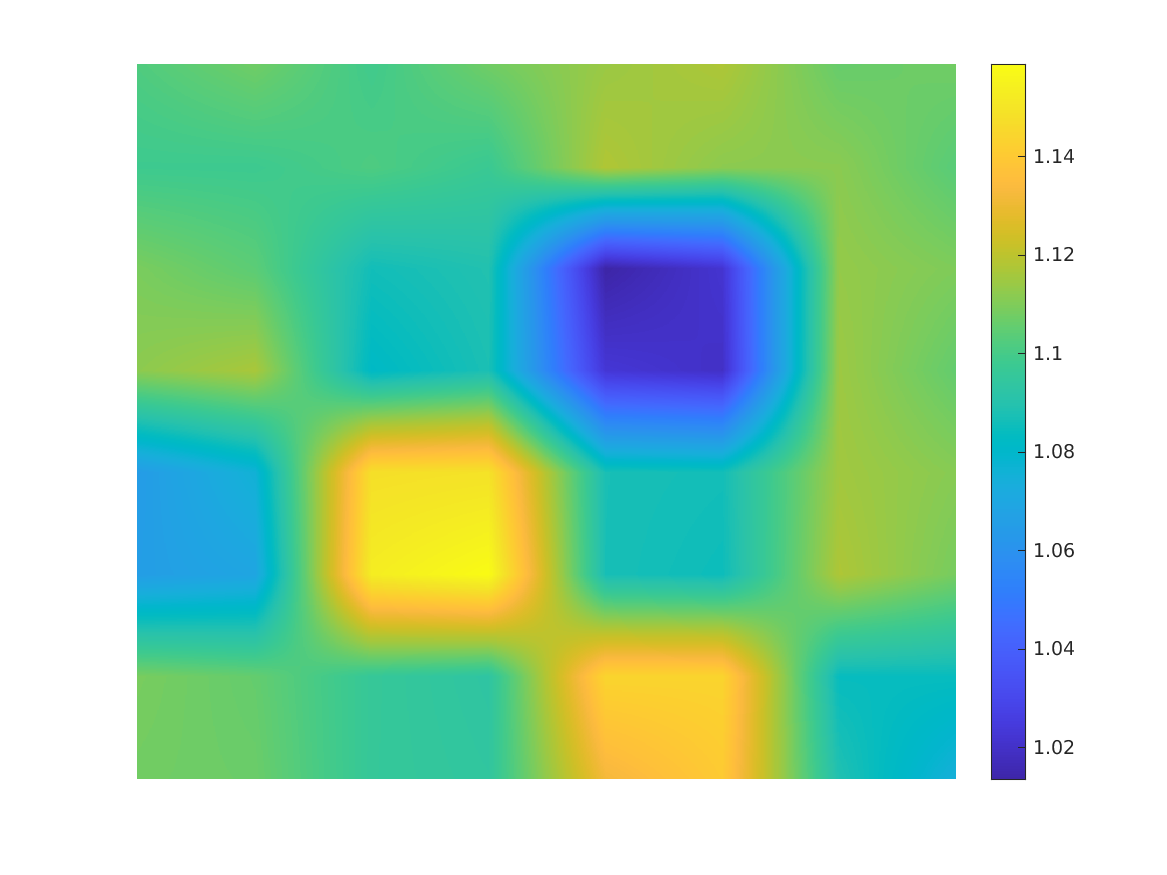} 
		\includegraphics[scale=0.35]{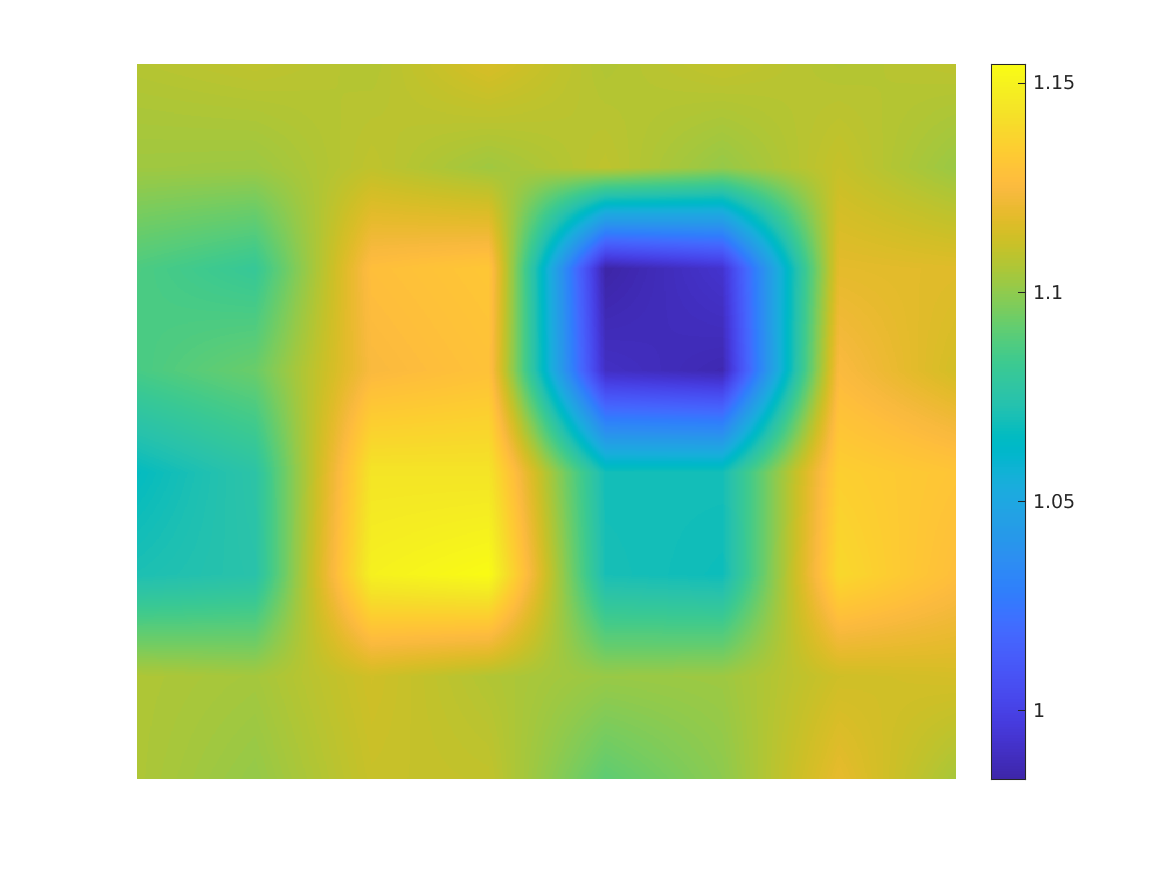} 
		\caption{Reconstructions by sparse adaptive interpolation MCMC with wavelet prior using 64 pixels. Both plots use $10^6$ samples and $5000$ interpolating nodes. Left: 16 electrodes. Right: 64 electrodes.}
		\label{fig:64pix}
	\end{figure}

	Finally, we consider a reconstruction using trigonometric prior. 
	In this experiment, we use $K=16$ electrodes. 
	%We fix $\eta=10^{14}, \tau =30, \gamma' = 10$ in \eqref{Matern_para}.  
	We replace the expansion \eqref{Matern_para} by a finite sum with 
	$(j_1,j_2)\in \{0,1,\dots,7\}$, corresponding to $J=64$ parameters. 
	We generate data by using a realization from the uniform distribution on $[-1,1]^{64}$ for the $y_{(j_1,j_2)}$'s. The target reference conductivity and the reconstruction by sparse adaptive interpolation MCMC with $M=10^6$ samples are shown in  \cref{fig:matern_recovery}. This example illustrates the flexibility of the sparse adaptive interpolation MCMC process as it can be applied to different types of parametrization of the conductivity.
	
	\begin{figure}
		\centering
		\includegraphics[scale=0.35]{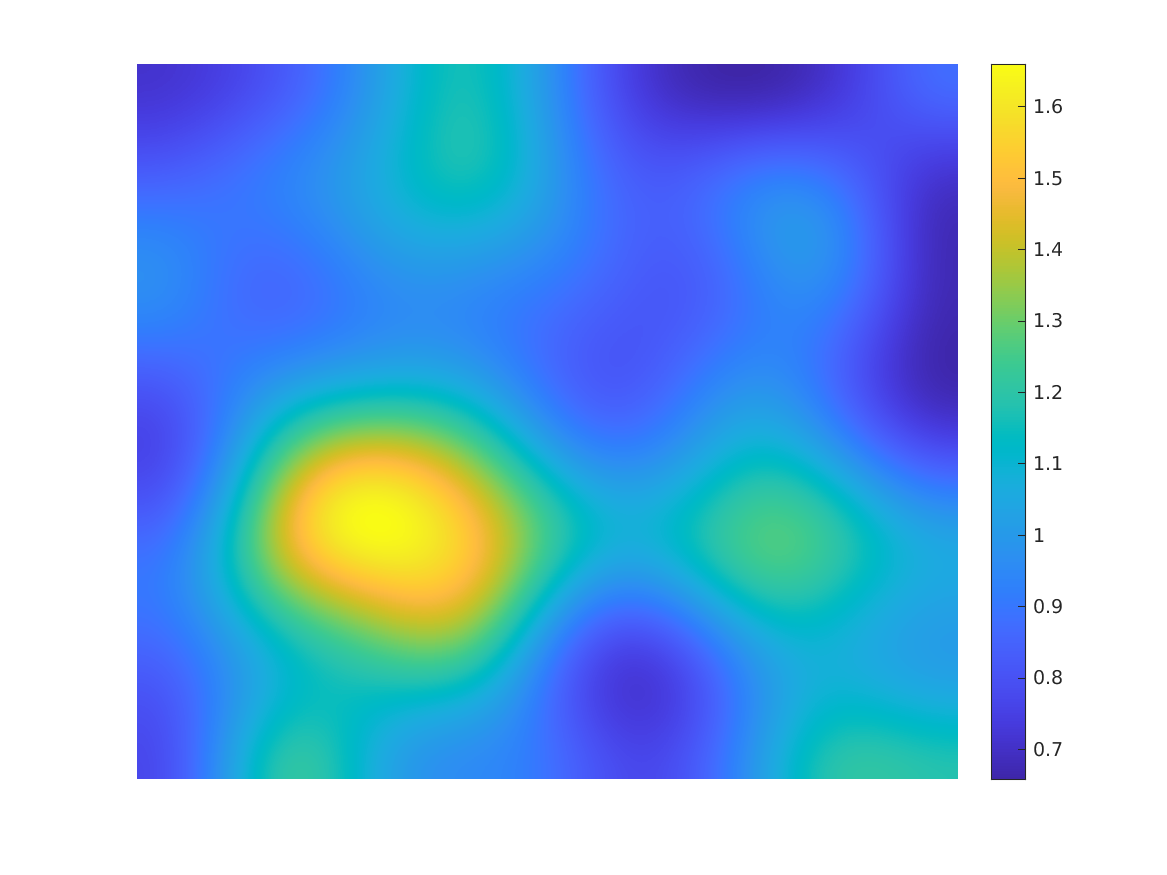} 
		\includegraphics[scale=0.35]{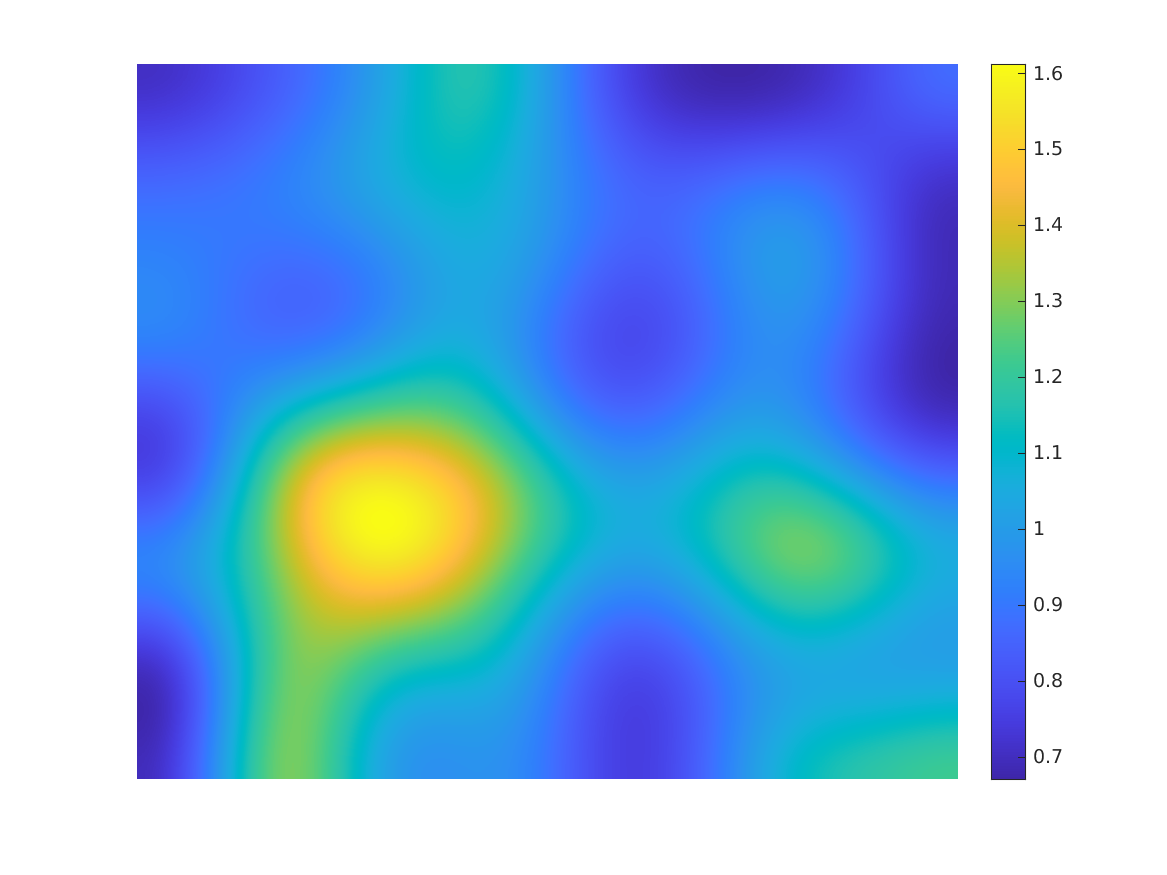} 
		\caption{Reconstruction with trigonometric prior. Both plots use $10^6$ samples and $5000$ interpolating nodes. Left: truth. Right: reconstruction.}
		\label{fig:matern_recovery}
	\end{figure}

%	\section{Conclusions}
%	We have proposed the algorithm SAI-MCMC and demonstrated its viability to accelerate the calculation of MCMC for the EIT problem. Our numerical examples with simulated data show the possibility of reconstructing the unknown conductivity using our proposed algorithm.  To consider more realistic experiments, we may employ priors such as level-set and star-shaped \cite{Dunlop2016} to incorporate information about the number of inclusions and their positions. In addition, the quantities $\gamma, \gamma'$ in the wavelet and trigonometric priors plays could be learned by treating them as random unknowns, using the hierarchical Bayesian approach in \cite{Dunlop2016H}. 
	
%	On the theoretical aspect, we have shown the existence of a sequence of index sets $(\Lambda_N)$ such that the error estimate in Proposition \ref{G-Gl} holds. In order to analyze the complexity of SAI-MCMC, we made the assumption it satisfies a similar error estimate. However, as the adaptive selection of index sets is a greedy algorithm, justification of such an assumption is a challenging task. Recently, new algorithms on adaptive sparse grids and their error estimates have been studied in the context of parametric elliptic PDEs in \cite{Eigel2022}.  

	\section*{Acknowledgements}
The first author gratefully acknowledges a postgraduate scholarship awarded by the Singapore Ministry of Education, and financial support from Nanyang Technological University. The research is supported by the Singapore Ministry of Education Tier 2 grant MOE2017-T2-2-144.

	\bibliographystyle{plain}
	
	\bibliography{EIT_uniform_bib}

	\appendix

	\section{Regularity and FE approximation of the forward solution}\label{appendix:regularity}
	
	%We are going to prove Proposition \ref{V-Vl}, which estimate the FE error of the forward solution.  In Proposition \ref{v_H2_bound}, we provide an explicit bound for $||v||_{H^2(D)/\mathbb{R}}$ in terms of $\sigma$.  Based on this bound, we obtain a version of Cea's lemma. We finish the proof of Proposition \ref{V-Vl} by using the duality argument. 
	 We derive in this appendix the FE error of the forward solver, uniformly with respect to $\by\in U$. We thus need to establish the regularity  of the solution to the forward equation. We have the following estimate.

	\begin{lemma}\label{v<=(v,V)} 
		For all $[(v,V)]\in \mathcal{H}^1$, 
		\begin{equation*}
			||\nabla v||_{L^2(D)}\le ||(v,V)||_{\mathcal{H}^1}.
		\end{equation*}
	\end{lemma}
	\begin{proof}
		For every $c\in \mathbb{R}$,  
		\begin{equation*}
			||\nabla v||_{L^2(D)}=||\nabla (v-c)||_{L^2(D)}\le ||v-c||_{H^1(D)}\le ||v-c||_{H^1(D)}+|V-c{\bf 1}|.
		\end{equation*}
		Taking infimum with respect to $c\in \mathbb{R}$, we get the conclusion.
	\end{proof}

	%%%%%%%%%%%%%%%%%%%%%%%%%%%%%%%%%%%%%%%%%%%%%%%%%%%%%%
%	\begin{comment}
%		{\todo if this regularity result is standard somewhere (in Hyvonen?), cite it instead}
%		\textcolor{blue}{Hyvonen's paper only shows that if $\zeta\in H^s(\partial D)$ then $v\in H^{s+3/2}(D)$. But for our purpose, it is crucial to specify the dependence of the RHS'constant on $\sigma$. We will need to use this estimate again for the lognormal case.}
%	\end{comment}
%%%%%%%%%%%%%%%%%%%%%%%%%%%%%%%%%%%%%%%%%%%%%%%%%%%%%%%%%%%%%%
%	\begin{assumption}\label{assumption:W1inf}
%		$\sigma\in W^{1,\infty}(D)$ and $\zeta\in H^t(\partial D)$ for some $t>\frac{1}{2}$. 
%	\end{assumption} 
%	As in \cite{Hyvonen2017}, if Assumption \ref{assumption:W1inf} is satisfied, $v\in H^2(D)/\mathbb{C}$. In the following, we provide an explicit bound on the $H^2$-norm of $v$ in terms of $\sigma$. 
In the following proposition, we show the uniform $H^2(D)$ regularity for the solution $v(\by)$ of the forward equation, uniformly with respect to $\by\in U$.
	\begin{proposition}\label{v_H2_bound} 
		Under Assumption \ref{assumption:sigma,zeta}, for the solution $(v,V)$ to \eqref{SCEM}  $\|v(\by)\|_{H^2(D)/{\mathbb R}}$ is uniformly bounded with respect to $\by\in U$.
%		\begin{equation*}
%			||v||_{H^2(D)/\mathbb{R}}\le C\frac{||\sigma||_{W^{1,\infty}(D)}+\sigma^-||\zeta||_{{\ha W^{1,\infty}}(\partial(D))}||\sigma^{-1}||_{W^{1,\infty}(D)}}{\sigma^-\min\{\sigma^-,\zeta^-\}}|I|.
%		\end{equation*}
	\end{proposition}
	\begin{proof}
		From \eqref{SCEM}, 
		\begin{equation*}
			\nu\cdot \sigma \nabla v=\zeta(V-v)\;\text{in }H^{-1/2}(\partial D), 
		\end{equation*}
		(we omit the $\by$ dependence of $(v,V)$ and $\sigma$ for conciseness).
		Since $\sigma\in W^{1,\infty}(D)$ and $\sigma^{-}>0$, we have $1/\sigma\in W^{1,\infty}(D)$. It follows that $\frac{1}{\sigma}w\in H^1(D)$ for all $w\in H^1(D)$. Hence, %from the definition of normal derivative via Green formula, 
		\begin{equation*}
			\nu\cdot \nabla v=\frac{\zeta(V-v)}{\sigma}.
		\end{equation*}
		Together with the identity $\nabla\cdot (\sigma \nabla v) = \nabla \sigma \cdot \nabla v + \sigma \Delta v$, we deduce that $v$ satisfies  
		\begin{equation*}
			\begin{cases}
				\Delta v= -\frac{\nabla \sigma \cdot \nabla v}{\sigma}\ \mbox{in } D,\\
				\nu\cdot \nabla v=\frac{\zeta(V-v)}{\sigma}\ \mbox{on }\partial D. 
			\end{cases}
		\end{equation*} 
		As $V-v\in H^1(D)$, by the trace theorem $V-v\in H^{1/2}(\partial D)$. %Since $\zeta\in H^t(\partial D)$ with $t> \frac{1}{2}$, Lemma 2.3 in \cite{Hyvonen2017} implies $\zeta(V-v)\in H^{1/2}(\partial D)$. Since $1/\sigma\in W^{1,\infty}(D)$, we have $\frac{\zeta(V-v)}{\sigma}\in H^{1/2}(\partial D)$. Hence we obtain the following $H^2$-estimate of the above Neumann problem \cite{Lions1972}
 Since $\zeta\in W^{1,\infty}(\partial D)$, $\zeta(V-v)\in H^{1/2}(D)$. Since $D$ is convex, the solution $v_0$ of the problem
\begin{equation*}
			\begin{cases}
				\Delta v_0= -\frac{\nabla \sigma \cdot \nabla v}{\sigma}{\ \mbox{in } D,}\\
				\nu\cdot \nabla v_0=0 {\ \mbox{on }\partial D},
			\end{cases}
		\end{equation*} 
	%	\begin{equation*}
belongs to $H^2(D)/\mathbb{R}$, with
\begin{align*}
	|v_0\|_{H^2(D)/\mathbb{R}}\le C\|\frac{\nabla \sigma \cdot \nabla v}{\sigma}\|_{L^2(D)}
	\le C\frac{||\sigma||_{W^{1,\infty}(D)}||\nabla v||_{L^2(D)}}{\sigma^{-}}\le C|I|,
\end{align*}
which is uniformly bounded with respect to $\by\in U$
 (see, e.g, Grisvard \cite{Grisvard2011} Theorem 3.2.1.3). Let $\zeta_k$ be the restriction of $\zeta$ on $E_k$, $\zeta_k$ is 0 outside $E_k$.  Let $\partial D_k$ be the affine edge/surface of $\partial D$ that contains $E_k$ in $\mathbb{R}^d$, we extend $\zeta_k(V-v)/\sigma$ to 0 outside $E_k$. As $\bar E_k$ belongs to the interior of $\partial D_k$, we can find a function $\phi_k\in C^\infty_0(\mathbb{R}^d)$ such that $\phi_k(x)=1$ for $x\in E_k$ and $(\partial D\setminus\partial D_k)\cap {\rm supp}(\phi_k)=\emptyset$. From the trace theorem, as $\partial D_k$ is linear, we can find a function $w'_k\in H^2(D\cap{\rm supp}(\phi_k))$ such that $w_k'=0$ on $\partial D_k$ and $\nu\cdot\nabla w_k'=\zeta_k(V-v)/\sigma$ on $\partial D_k$ (which is 0 outside $E_k$).  We note that $\|w_k'\|_{H^2(D)}$ is uniformly bounded with respect to $\by$, namely 
\begin{equation*}
	\|w_k'\|_{H^2(D)}\le C|\zeta|_{W^{1,\infty}(\partial D)}||V-v||_{H^1(D)}\le C||(v,V)||_{\mathcal{H}^1}\le C|I|.
\end{equation*}
Let $w_k=\phi_kw_k'$, we find that $w_k\in H^2(D)$ such that on $\partial D_k$, $w_k=0$ and 
\[
\nu\cdot\nabla w_k=\nu\cdot(\nabla\phi_k w_k'+\phi_k\nabla w_k')=\phi_k\zeta_k(V-v)/\sigma
\]
which equals 0 outside $E_k$ and $\zeta_k(V-v)/\sigma$ on $E_k$, as $\phi_k=1$ on $E_k$ and $\zeta_k=0$ outside $E_k$. Further, as other edges/surfaces of $\partial D$ are outside the support of $w_k$, we find that $\nu\cdot\nabla w_k=0$ on other edges/surfaces. We consider the problem
\begin{equation*}
			\begin{cases}
				\Delta v_k= 0{\ \mbox{in } D,}\\
				\nu\cdot \nabla v_k={\zeta_k(V-v)\over\sigma}=\nu\cdot\nabla w_k {\ \mbox{on }\partial D},
			\end{cases}
\end{equation*}
We note that $v_k-w_k$ is the solution of the homogeneous Neunman problem
\begin{equation*}
			\begin{cases}
				\Delta (v_k-w_k)= -\Delta w_k\in L^2(D)\ \mbox{in } D,\\
				\nu\cdot \nabla (v_k-w_k)=0 {\ \mbox{on }\partial D},
			\end{cases}
		\end{equation*} 
 which is uniformly bounded in  $H^2(D)/{\mathbb R}$ as $D$ is convex and $\Delta w_k$ is uniformly bounded in $L^2(D)$ for all $\by$ (Theorem 3.2.1.3 of \cite{Grisvard2011}). By superposition, we obtain the desired result.
\end{proof}
%
%			||v||_{H^2(D)/\mathbb{C}}\le C\left(\left|\left|\frac{\nabla \sigma \cdot \nabla v}{\sigma}\right|\right|_{L^2(D)}+\left|\left| \frac{\zeta(V-v)}{\sigma}\right|\right|_{H^{1/2}(\partial D)}\right). 
%		\end{equation*}
%		We proceed to estimate the terms on the RHS. Since $\sigma\in W^{1,\infty}(D)$ and $\sigma^{-}>0$, we have
%		\begin{equation*}
%			\left|\left|\frac{\nabla \sigma \cdot \nabla v}{\sigma}\right|\right|_{L^2(D)}\le \frac{||\sigma||_{W^{1,\infty}}||\nabla v||_{L^2(D)}}{\sigma^-}. 
%		\end{equation*}
%		By Lemma \ref{v<=(v,V)},  
%		\begin{equation}\label{grad_sigma_L2}
%			\left|\left|\frac{\nabla \sigma \cdot \nabla v}{\sigma}\right|\right|_{L^2(D)}\le \frac{||\sigma||_{W^{1,\infty}}}{\sigma^-}||(v,V)||_{%\mathcal{H}^1}.
%		\end{equation}
%		On the other hand, since $\sigma^{-1}\in W^{1,\infty}(D)$, we get
%		\begin{equation*}
%			\left|\left| \frac{\zeta(V-v)}{\sigma}\right|\right|_{H^{1/2}(\partial D)}\le C||\sigma^{-1}||_{W^{1,\infty}(D)}||\zeta(V-v)||_{H^{1/2}(\partial D)}
%		\end{equation*}
%		From Lemma 2.3 and Lemma 2.4 in \cite{Hyvonen2017},  
%		\begin{align*}
%			||\zeta(V-v)||_{H^{1/2}(\partial D)}&\le C||\zeta||_{H^{s}(\partial D)}||V-v||_{H^{1/2}(\partial D)}\\
%			&\le  C||\zeta||_{H^{s}(\partial D)}||(v,V)||_{\mathcal{H}^1}.
%		\end{align*}
%		This estimate together with \eqref{H1_norm_bound} and \eqref{grad_sigma_L2} yield the desired result. 	\end{proof}
	
	Let $(v^l,V^l)$ be the FE approximation of $(v,V)$ with mesh size $O(2^{-l})$.

	\begin{lemma}
		Under Assumption \ref{assumption:sigma,zeta}, 
		\begin{equation}\label{(v-vl,V-Vl)_bound}
			||(v,V)-(v^l,V^l)||_{\mathcal{H}^1}\le C 2^{-l}.%\max\{\sigma^+,\zeta^+\}\frac{||\sigma||_{W^{1,\infty}(D)}+\sigma^{-}||\zeta||_{H^s(\partial(D))}||\sigma^{-1}||_{W^{1,\infty}(D)}}{\sigma^-\min\{\sigma^-,\zeta^-\}^2}|I|.
		\end{equation}
	\end{lemma}
	\begin{proof}
		By Cea's lemma,  
		\begin{equation*}
			||(v,V)-(v^l,V^l)||_{\mathcal{H}^1}\le C\frac{\max\{\sigma^+,\zeta^+\}}{\min\{\sigma^-,\zeta^-\}}\inf_{(w^l,W^l)\in\mathcal{V}^l}||(v,V)-(w^l,W^l)||_{\mathcal{H}^1}.
		\end{equation*}
		From the definition of $\mathcal{H}^1$-norm,  
		\begin{align*}
			\inf_{(w^l,W^l)\in\mathcal{V}^l}||(v,V)-(w^l,W^l)||_{\mathcal{H}^1} 
			& = \inf_{w^l\in P^l,W^l\in \mathbb{R}^M,c\in \mathbb{R}} \left( ||v-w^l-c||_{H^1(D)}+|V-W^l-c{\bf 1}|\right) \\
			&\le \inf_{w^l\in P^l,c\in \mathbb{R}}||(v-c)-w^l||_{H^1(D)}. 
		\end{align*}
		The desired estimate follows from Proposition \ref{v_H2_bound} and the approximation property of $P^l$, namely, 
		\begin{equation}\label{FE_2^-l}
			\inf_{w^l\in P^l}||(v-c)-w^l||_{H^1(D)}\le C2^{-l}||v-c||_{H^2(D)}.
		\end{equation}
	\end{proof}
	
	\begin{proposition}\label{V-Vl}
		Under Assumption  \ref{assumption:sigma,zeta},  
		\begin{equation*}
			|V-V^l|\le C2^{-2l}.%(\max\{\sigma^+,\zeta^+\})^3C_1(\sigma,\zeta)^2|I|, 
		\end{equation*}
%		where 
%		$C_1(\sigma,\zeta)=\frac{||\sigma||_{W^{1,\infty}(D)}+\sigma^{-}||\zeta||_{H^s(\partial(D))}||\sigma^{-1}||_{W^{1,\infty}(D)}}{\sigma^-\min\{\sigma^-,\zeta^-\}^2}$. 
	\end{proposition}   

	\begin{proof}%[Proof of Proposition \ref{V-Vl}]
		We follow the Aubin-Nitsche duality argument. 
		For each $I^*\in \mathbb{R}^K_\diamond$, let $(v_{I^*},V_{I^*})\in \mathcal{H}^1$ be the solution to the problem
		\begin{equation*}
			B((v_{I^*},V_{I^*}),(w,W))=I^*\cdot W\quad \forall (w,W)\in \mathcal{H}^1
		\end{equation*}
		and $(v_{I^*}^l,V_{I^*}^l) \in \mathcal{V}^L$ the solution to 
		\begin{equation*}
			B((v_{I^*}^l,V_{I^*}^l),(w^l,W^l))=I^*\cdot W^l\quad \forall (w^l,W^l)\in \mathcal{V}^l. 
		\end{equation*}
		We have
		\begin{align*}
			|I^*\cdot (V-V^l)| &=|B((v_{I^*},V_{I^*}),(v-v^l,V-V^l))|\\
			&=|B((v_{I^*}-v_{I^*}^l,V_{I^*}-V_{I^*}^l),(v-v^l,V-V^l))|\\
			&\le C\max\{\sigma^+,\zeta^+\}||(v_{I^*}-v_{I^*}^l,V_{I^*}-V_{I^*}^l)||_{\mathcal{H}^1}||(v-v^l,V-V^l)||_{\mathcal{H}^1}\\
			&\le C2^{-2l} \|v_{I^*}\|_{H^2(D)/{\mathbb R}}\|v\|_{H^2(D)/{\mathbb R}}%(\max\{\sigma^+,\zeta^+\})^3C_1(\sigma,\zeta)^2|I^*|\times|I|,
		\end{align*}
		where the third line follows from the boundedness of the bilinear form $B$ and the fourth line follows from \eqref{(v-vl,V-Vl)_bound}. 	
 When $|I^*|=1$, from the proof Proposition \ref{v_H2_bound}, $\|v_{I^*}\|_{H^2(D)/{\mathbb R}}$ is uniformly bounded with respect to $I^*$.	Since $\mathbb{R}^K_\diamond$ is a closed subspace of $\mathbb{R}^K$, we obtain  
		\begin{equation*}
			|V-V^l|=\max_{{ I^*}\in \mathbb{R}^K_\diamond, |I^*|=1}|I^*\cdot (V-V^l)|\le C2^{-2l}.%(\max\{\sigma^+,\zeta^+\})^3C_1(\sigma,\zeta)^2|I|.
		\end{equation*}
	\end{proof}

	\section{Smoothened CEM with complex-valued conductivity}\label{appendix:complex_SCEM} 
To establish the convergence rate in Proposition \ref{best_N_existence}, following \cite{Chkifa2013}, we extend the range of the parameters to the complex plane, and consider a complex smoothened CEM problem. CEM problems with a complex-valued conductivity $\sigma$ have been studied extensive (see, e.g., \cite{Hyvonen2017} and references there in). %, we consider the model \eqref{SCEM} where $\sigma$ is complex-valued. 
  Our aim is to establish the analyticity of the solution, and bound the coefficients %We need this extension to establish a bound for the coefficient 
of the Taylor expansion of the solution of the complex parametric problem in Appendix \ref{section:error_estimate}. %We look for solutions in the space 
 The solution belongs to the space
	\begin{equation*}
		\mathcal{H}^1_c:=(H^1(D;\mathbb{C})\oplus \mathbb{C}^K)/\mathbb{C}. 
	\end{equation*}
	The corresponding variational formulation is 
	\begin{equation}\label{variational_complex}
		B((v,V),(w,W);\sigma)=I\cdot \overline{W}\quad \forall\;(w,W)\in \mathcal{H}^1_c, 
	\end{equation}
	where the sesquilinear form $B:\mathcal{H}^1_c\times \mathcal{H}^1_c\to \mathbb{C}$ is defined as 
	\begin{equation*}
		B((v,V),(w,W);\sigma)=\int_D \sigma \nabla v\cdot \nabla \overline{w} dx+\int_{\partial D}\zeta (V-v)(\overline{W}-\overline{w})dS.
	\end{equation*} 
	We denote by $\mathcal{A}(D)$ the class of $\sigma\in L^\infty(D;\mathbb{C})$ such that there exist $\sigma^+,\sigma^->0$ satisfying 
	\begin{equation}\label{A(D)}
		\sigma^-\le \operatorname{Re}(\sigma(x))\le |\sigma(x)|\le \sigma^+\; \forall x\in D. 
	\end{equation}
	To simplify the notation, we use $\mathcal{H}^1$  for $\mathcal{H}^1_c$. It should not be confused with the real vector space counterpart. 
	From Lemma 2.1 in \cite{Hyvonen2017}, provided that \eqref{A(D)} holds,  the sesquilinear form $B$ is bounded and coercive with 
	\begin{equation*}
		\operatorname{Re}B((v,V),(v,V))\ge C\min\{\sigma^-,\zeta^-\}||(v,V)||_{\mathcal{H}^1}. 
	\end{equation*}
%	where, as in Appendix \ref{appendix:regularity}, we denote by $C$ a generic constant depending only on the domain $D$ and the electrodes. 
	In the quotient space $\mathcal{H}^1$, we choose the representative $(v,V)$ such that $V$ belongs to 
	\begin{equation*}
		\mathbb{C}^K_{\diamond}:= \{U\in \mathbb{C}^K: U_1+\dots+U_K=0\}. 
	\end{equation*}  
	As in Lemmas \ref{V<=(v,V)} and \ref{v<=(v,V)}, we can show that $|V|\le ||(v,V)||_{\mathcal{H}^1}$ and $||\nabla v||_{L^2(D)}\le ||(v,V)||_{\mathcal{H}^1}$. Hence 
	\begin{equation*}
		|V|\le C\frac{|I|}{\min \{\sigma^-,\zeta^-\}}.  
	\end{equation*} 
	The piecewise linear FE approximation  $(v^l,V^l)$ with mesh size $O(2^{-l})$ belongs to the space $\mathcal{V}^l_C:=(P^l\oplus \mathbb{C}^K)/\mathbb{C}$. We have 
	\begin{equation}\label{Vl_bounded}
		|V^l|\le C\frac{|I|}{\min \{\sigma^-,\zeta^-\}}.
	\end{equation}

 We now show the continuous  dependence of the solution to the forward problem on the conductivity.	 
	\begin{lemma}\label{M1,2} 
%		Let $\sigma\in \mathcal{A}(D)$ and $(\sigma_\epsilon)_{\epsilon \in \mathbb{C}}\subset \mathcal{A}(D)$ such that $\lim_{\epsilon\to 0} ||\sigma_\epsilon-\sigma||_{L^\infty(D)}=0$. Then 
 Let $\sigma_1,\sigma_2$ belong to $\mathcal{A}(D)$ with bounds $\sigma^{-}$ and $\sigma^{+}$ in \eqref{A(D)}. We then have
		\begin{equation*}
			||(v(\sigma_1)-v(\sigma_2),V(\sigma_1)-V(\sigma_2))||_{\mathcal{H}^1}\le C\|\sigma_1-\sigma_2\|_{L^\infty(D)},
	\end{equation*}
where $C$ depends on $\sigma^-$ and $\sigma^+$. 

	\end{lemma}

	%{\todo avoid using $h$ as we normally use it to denote the mesh size} \textcolor{blue}{$h$ is changed to $\epsilon$}
	\begin{proof} %[\textbf{Proof of Lemma \ref{M1,2}}]
		For any $\sigma\in\mathcal{A}(D)$, denote $\mathcal{M}(\sigma):=((v(\sigma),V(\sigma)))$. 
 We show that for $\sigma_1,\sigma_2\in \mathcal{A}(D)$, 
		\begin{equation}\label{M1-M2}
			||\mathcal{M}(\sigma_1)-\mathcal{M}(\sigma_2)||_{\mathcal{H}^1}\le C\frac{||\sigma_1-\sigma_2||_{L^\infty(D)}||\mathcal{M}(\sigma_2)||_{\mathcal{H}^1}}{\min\{\inf_{x\in D}\operatorname{Re}(\sigma_1(x)),\zeta^-\}}.
		\end{equation} 
		Denote $(v_j, V_j)=(v(\sigma_j),V(\sigma_j)), j=1,2$. Since for all $(w,W)\in \mathcal{H}^1$, 
		\begin{equation*}
			B((v_1, V_1),(w,W);\sigma_1)=I\cdot \overline{W}=B((v_2,V_2),(w,W);\sigma_2),
		\end{equation*}
		we have 
		\begin{align*}
			0&=B((v_1,V_1),(w,W);\sigma_1))-B((v_2,V_2),(w,W);\sigma_1)+B((v_2,V_2),(w,W);\sigma_1)-B((v_2,V_2),(w,W);\sigma_2)\\
			&=\int_D \sigma_1 \nabla (v_1-v_2) \cdot \nabla \overline{w}dx+\int_{\partial D}\zeta((V_1-V_2)-(v_1-v_2))(\overline{W}-\overline{w})dS
			+\int_D (\sigma_1-\sigma_2) \nabla v_2 \cdot \nabla \overline{w}dx.
		\end{align*}
		Let $(w,W)=(v_1-v_2,V_1-V_2)$. Then
		\begin{equation*}
			\int_D \operatorname{Re}(\sigma_1) |\nabla (v_1-v_2)|^2dx+\int_{\partial D} \zeta |(V_1-V_2)-(v_1-v_2)|^2dS\le \int_D |\sigma_1-\sigma_2||\nabla v_2\cdot \nabla (\overline{v_1-v_2})|dx.
		\end{equation*}
		From estimate (9) in \cite{Hyvonen2017}, we deduce  
		\begin{equation*}
			||(v_1-v_2,V_1-V_2)||_{\mathcal{H}^1}^2\le C(||\nabla (v_1-v_2)||^2_{L^2(D)}+||(V_1-V_2)-(v_1-v_2)||_{L^2(\partial D)}).
		\end{equation*}
		From $\eqref{zeta_lower_bound}$ and the assumption that $\operatorname{Re}(\sigma_1)$ is bounded below, 
		\begin{equation*}
			||(v_1-v_2,V_1-V_2)||_{\mathcal{H}^1}^2\le C_{\sigma_1}\left( \int_D \operatorname{Re}(\sigma_1) |\nabla (v_1-v_2)|^2dx+\int_{\partial D}\zeta ((V_1-V_2)-(v_1-v_2))^2dS\right),
		\end{equation*}
		where $C_{\sigma_1}=\frac{C}{\min\{\inf_{x\in D}\operatorname{Re}(\sigma_1(x)),\zeta^-\}}$. 
		Hence 
		\begin{align*}
			||(v_1-v_2,V_1-V_2)||_{\mathcal{H}^1}^2&\le C_{\sigma_1}\left(\int_D|\sigma_1-\sigma_2|^2|\nabla v_2|^2dx\right)^{1/2}||\nabla (v_1-v_2)||_{L^2(D)}\\
			&\le C_{\sigma_1}||\sigma_1-\sigma_2||_{L^\infty(D)}||\nabla v_2||_{L^2(D)}||(v_1-v_2,V_1-V_2)||_{\mathcal{H}^1}\\
			&\le C_{\sigma_1}||\sigma_1-\sigma_2||_{L^\infty(D)}||(v_2,V_2)||_{\mathcal{H}^1}||(v_1-v_2,V_1-V_2)||_{\mathcal{H}^1}.
		\end{align*}
		We then get \eqref{M1-M2}.  From the uniform boundedness of $\|\mathcal{M}(\sigma)\|_{\mathcal{H}^1}$  for all $\sigma \in A(D)$, we get the conclusion.
		
%		Since $\lim_{\epsilon\to 0} ||\sigma_\epsilon - \sigma||_{L^\infty(D)}=0$ and $\inf_{x\in D} \operatorname{Re}(\sigma(x))>0$, there exists a constant $c>0$ such that for $|\epsilon|$ small enough, 
%		\begin{equation*}
%			\operatorname{Re}(\sigma_\epsilon)\ge c\text{ a.e.} 
%		\end{equation*}
%		Taking $\sigma_1 = \sigma_\epsilon$, $\sigma_2 = \sigma$ in \eqref{M1-M2}, we get the desired result. 
		
	\end{proof}

	\section{Error estimates}\label{section:error_estimate}
	\begin{comment}
		To simplify the analysis, we only consider the affine parametrization \eqref{affine_para}, i.e., 
		\begin{equation*}
			\sigma^J(x,y)=\bar{\sigma}(x)+\sum_{j=1}^{J}y_j\psi_j(x).
		\end{equation*} 
		Throughout the remainder of the appendix, we denote $U_J:=[-1,1]^J$.  
	\end{comment}
	
	\subsection{Truncation error}
	
	\begin{proof}[\textbf{Proof of Proposition \ref{mu-muJ}}]
	
		The normalizing constants of $\mu^\delta$ and $\mu^{J,\delta}$ in \eqref{eq:mudelta} and \eqref{eq:muJ} are 
		\begin{equation*}
			Z(\delta):=\int_{U} \exp(-\Phi(\by;\delta)) d\mu_0(\by);\quad Z^{J}(\delta):=\int_{U} \exp(-\Phi^{J}(\by;\delta)) d\mu_0(\by).
		\end{equation*}
		From the definition of Hellinger distance, 
		\begin{align*}
			2d_{Hell}(\mu^{\delta},\mu^{J,\delta})^2
			&= \int_{U} \left(Z(\delta)^{-1/2}\exp(-\frac{1}{2}\Phi(y;\delta))-(Z^{J}(\delta))^{-1/2}\exp(-\frac{1}{2}\Phi^{J}(y;\delta))\right)^2d\mu_0(\by)\\
			&\le I_1+I_2, 
		\end{align*}
		where 
		\begin{align*}
			I_1 & = \frac{1}{Z(\delta)}\int_{U} \left(\exp(-\frac{1}{2}\Phi(\by;\delta))-\exp(-\frac{1}{2}\Phi^{J}(\by;\delta))  \right)^2d\mu_0(\by)\\
			I_2 & = |Z(\delta)^{-1/2}-Z^{J}(\delta)^{-1/2}|^2 Z^{J}(\delta).
		\end{align*}
		 From \eqref{G_bounded} and \eqref{GJ_bounded}, $Z(\delta)$ and $Z^{J}(\delta)$ are uniformly bounded below by a positive constant. Hence
		\begin{align*}
			I_2&\le C|Z(\delta)-Z^{J}(\delta)|^2\\
			&\le C\int_{U} |\exp(-\Phi(\by;\delta))-\exp(-\Phi^{J}(\by;\delta)|^2d\mu_0(\by).
		\end{align*}
		By the inequality $|e^{-a}-e^{-b}|\le |a-b|$ for $a,b\ge 0$, we deduce 
		\begin{equation*}
			I_1,I_2\le C\left(\sup_{\by\in U}|\Phi(\by;\delta)-\Phi^{J}(\by;\delta)|\right)^2.
		\end{equation*}
		Together with \eqref{phi-phi_J}, we arrive at the conclusion. 
	\end{proof}

	\subsection{Interpolation error}
	The idea of using analyticity to obtain convergence rates of polynomial approximation of parametric PDEs has been exploited in \cite{Cohen2011} and applied to interpolation in \cite{Chkifa2013, Chkifa2015}. As the results in these papers are formulated for elliptic PDEs with Dirichlet boundary conditions, minor modifications are needed to apply to EIT. %In addition, as the forward map in our situation is finite-dimensional, some technicalities of infinite-dimensional analysis can be omitted.  

	We begin with the analyticity of the parametric forward solutions. The parametrization of $\sigma$ can be extended to the complex plane by defining 
	\begin{equation*}
		\sigma^J(x,\xi):=\bar{\sigma}(x)+\sum_{j=1}^{J}\xi_j\psi_j(x), \quad \xi\in \mathbb{C}^J. 
	\end{equation*}
	For any sequence of positive numbers $\rho=(\rho_1,\dots,\rho_J)$, let 
	$\mathcal{D}^J_\rho:=\{\xi\in \mathbb{C}^J: |\xi_j|\le \rho_j\; \forall j \}$. We note that $\mathcal{D}^J_\rho$ contains $U_J=[-1,1]^J$ when $\rho_j>1\;\forall j$. 
	If 
	\begin{equation}\label{admissible_rho}
		\sum_{j=1}^J \rho_j|\psi_j(x)|\le \bar{\sigma}(x)-\sigma^-/2, 
	\end{equation}
	then the polydisc $\mathcal{D}^J_\rho$ is contained in the set 
	\begin{equation*}
		\mathcal{A}_J:=\{\xi\in \mathbb{C}^J: 	\sigma^-/2\le \operatorname{Re}(\sigma^J(x,\xi))\le|\sigma^J(x,\xi)|\le  2\sigma^+ \; \forall x\in D \}.
	\end{equation*}
	It follows from Appendix \ref{appendix:complex_SCEM} that $\mathcal{G}^{J,l}(z)$ is well-defined for all $\xi\in \mathcal{A}_J$. 
	
	\begin{lemma}
		The function $\xi\mapsto \mathcal{G}^{J,l}(\xi)$ is holomorphic on $\mathcal{A}_J$.
	\end{lemma}
	\begin{proof}
		It suffices to show that for any $k=1,\dots,K-1$, the map $\xi\mapsto V^{l,(k)}(\sigma^J(\cdot,\xi))$  for the FE solution admits complex partial derivatives at every point in $\mathcal{A}_J$. 
		To ease notation, we drop the superscript $(l,k)$ and denote each of these maps by $V$ and their corresponding potential by $v$. Let $\{e_1,\dots,e_J\}$ be the canonical basis in $\mathbb{R}^J$. Fix $j\in \{1,\dots,J\}$ and $\xi\in \mathcal{A}_J$. For $h\in \mathbb{C}\backslash \{0\}$ with $|h|$ small enough such that $(v(\xi+he_j),V(\xi+he_j))$ is well-defined, let 
		\begin{align*}
			v_h(\xi)&:=\frac{v(\xi+he_j)-v(\xi)}{h};\\ V_h(\xi)&:=\frac{V(\xi+he_j)-V(\xi)}{h}.
		\end{align*}  
		Note that $(v(\xi),V(\xi))$ and $(v(\xi+he_j),V(\xi+he_j))$ satisfy the variational problems 
		\begin{equation*}
			B((v(\xi),V(\xi)),(w^l,W^l);\sigma^J(\xi)) = I\cdot W^l \quad \forall (w^l,W^l)\in \mathcal{V}^l
		\end{equation*}
		and 
		\begin{equation*}
			B((v(\xi),V(\xi+he_j)),(w^l,W^l);\sigma^J(\xi+he_j)) = I\cdot W^l \quad \forall (w^l,W^l)\in \mathcal{V}^l.
		\end{equation*}
		Subtracting these equalities, we find that for all $(w^l,W^l)\in \mathcal{V}^l$, 
		\begin{align*}
			0
			&=\int_D [\sigma(x,\xi+he_j)-\sigma(x,\xi)]\nabla v(\xi+he_j)\cdot \nabla \overline{w^l}(x) dx+\int_D \sigma(x,\xi)[\nabla v(\xi+he_j)-\nabla v(\xi)]\cdot \nabla \overline{w^l}(x) \\
			&+\int_{\partial D}\zeta[V(\xi+he_j)-V(\xi)-v(\xi+he_j)+v(\xi)]\overline{W^l-w^l}dS\\
			&=\int_Dh\psi_j(x) \nabla v(x,\xi+he_j)\cdot\nabla \overline{w^{l}}(x) dx+\int_D \sigma(x,\xi)h\nabla v_h(\xi)\cdot \nabla \overline{w^l}(x)dx\\
			&+\int_{\partial D}\zeta h[V_h(\xi)-v_h(\xi)]\overline{W^l-w^l}dS\\
			&=hB((v_h,V_h),(w^l,W^l);\sigma(\xi))+h\int_D\psi_j(x) \nabla v(x,\xi+he_j)\cdot\nabla \overline{w^l}(x) dx.
		\end{align*}
		It follows that $(v_h,V_h)$ satisfies 
		\begin{equation}\label{(vh,Vh)quotient}
			B((v_h,V_h),(w^l,W^l);\sigma(\xi))=-\int_D\psi_j(x) \nabla v(x,\xi+he_j)\cdot\nabla \overline{w^l}(x) dx. 
		\end{equation}
		For each $h\in \mathbb{C}$ and $(w,W)\in \mathcal{H}^1$, let 
		\begin{equation*}
			L_h((w,W)):=-\int_D\psi_j(x) \nabla v(x,\xi+he_j)\cdot\nabla \overline{w^l}(x) dx. 
		\end{equation*} 
		Then $L_h$ is a bounded linear functional  on $\mathcal{H}^1$ because 
		\begin{equation*}
			|L_h((w,W))|\le ||\psi_j||_{L^\infty(D)} ||v(\cdot,\xi+he_j)||_{H^1(D)}||w||_{H^1(D)}.
		\end{equation*}
		Let $(v_0,V_0)$ be the unique solution to the variational problem
		\begin{equation*}
			B((v_0,V_0),(w^l,W^l);\sigma(z))=L_0((w^l,W^l))\;\forall (w^l,W^l)\in \mathcal{V}^l.
		\end{equation*}
		For all $(w,w)\in \mathcal{H}^1$, we have 
		\begin{align*}
			|L_h((w,W))-L_0((w,W))|&=|\int_D \psi_j(x)[\nabla v(x,\xi+he_j)-v(x,\xi)]\cdot \nabla \overline{w}(x)dx|\\
			&\le ||\psi_j||_{L^\infty(D)}||\nabla [v(\xi+he_j)-v(\xi)]||_{L^2(D)}||\nabla w||_{L^2(D)}.
		\end{align*}
		From Lemma \ref{v<=(v,V)} and  Lemma \ref{M1,2}, we deduce 
		\begin{equation*}
			\lim_{h\to 0}||\nabla [v(\xi+he_j)-v(z)]||_{L^2(D)}=0.
		\end{equation*}
		Since $L_h((w,W))\to L_0((w,W))$ for every $(w,W)\in \mathcal{H}^1$, it follows that $(v_h,V_h)$ converges to  $(v_0,V_0)$ in $\mathcal{H}^1$.  
		From Lemma \ref{V<=(v,V)}, we have $\lim_{h\to 0}|V_h- V_0|= 0$. Therefore $\partial_{z_j}V(z)=V_0$.
	\end{proof}

	\begin{lemma}\label{unconditional_Taylor}
		The Taylor series $\sum_{\nu\in \mathbb{N}_0^J}t_{\nu,J,l}y_1^{\nu_1}\dots y_J^{\nu_J}$ converges unconditionally to $\mathcal{G}^{J,l}$ on $U_J$, where 
		\begin{equation*}
			t_{\nu,J,l}:=\frac{1}{\nu!}\partial^{\nu}\mathcal{G}^{J,l}(0).
		\end{equation*} 
	\end{lemma}
	\begin{proof}
		From \eqref{sum_psi} we have  
		\begin{equation*}
			\sum_{j=1}^J |\psi_j(x)|\le \bar{\sigma}(x)-\sigma^-. 
		\end{equation*}
		Hence there exists a sequence $\rho\in \mathbb{R}^J$ such that $\rho_j>1$ for all $j$ and \eqref{admissible_rho} holds. This implies that the Taylor series of $\mathcal{G}^{J,l}(\xi)$ converges unconditionally for all $\xi$ in the interior of $\mathcal{D}^J_\rho$, which contains $U_J$. 
	\end{proof}

	\begin{lemma}\label{t_nu_bound} 
		For any positive sequence $\rho\in \mathbb{R}^J$ satisfying  \eqref{admissible_rho}, there is a constant $C$ independent of $\nu,J,l$ such that 
		\begin{equation*}
			|t_{\nu,J,l}|\le C\rho^{-\nu}.
		\end{equation*}
	\end{lemma}
	\begin{proof}
		Since the map $z\mapsto \mathcal{G}^{J,l}(z)$ is holomorphic on $\mathcal{D}_\rho^J$, the Cauchy's integral theorem (see e.g., Theorem 2.2.6 in \cite{Hormander1990}) gives
		\begin{equation*}
			t_{\nu,J,l}
			=\frac{1}{(2\pi i)^J}\int_{|\xi_1|=\rho_1}\dots\int_{|\xi_J|=\rho_J}\frac{\mathcal{G}^{J,l}(\xi_1,\dots,\xi_J)}{\xi_1^{\nu_1+1}\dots \xi_J^{\nu_J+1}}. 
		\end{equation*}
		From \eqref{Vl_bounded}, we deduce that $\mathcal{G}^{J,l}(\xi)$ is uniformly bounded for all $J,l \in \mathbb{N}, \xi\in \mathcal{D}^J_\rho$. Hence the desired estimate follows.

	\end{proof}

	\begin{proof}[\textbf{Proof of Proposition \ref{best_N_existence}}]
		Using Lemma 4.2 in \cite{Chkifa2013} and Lemma \ref{unconditional_Taylor}, the interpolation error in approximating $\mathcal{G}^{J,l}$ can be linked to the decay of its Taylor coefficients, namely 
		\begin{equation}\label{relate_interp_taylor}
			\sup_{y\in U_J}|\mathcal{G}^{J,l}(y)-I_{\Lambda}\mathcal{G}^{J,l}(y)|\le 2\sum_{\nu \notin \Lambda}p_{\nu,J}(\theta)|t_{\nu,J,l}|,
		\end{equation}
		where $p_{\nu,J}(\theta):=\prod_{j=1}^{J}(1+\nu_j)^{\theta+1}$. With \eqref{relate_interp_taylor} and Lemma \ref{t_nu_bound}, we can proceed as in the proof of Theorem 4.3 in \cite{Chkifa2013} to complete the proof of the theorem. 
	\end{proof}
	
	\begin{remark}
		For log-affine parametrization \eqref{log_affine_para}, we can perform error analysis using results in \cite{Chkifa2015}. The Taylor series of $\mathcal{G}^{J,l}$ should be replaced by a Legendre series, and a similar interpolation error of order $(N+1)^{-s}$ could be obtained. 
	\end{remark}

	\subsection{Monte Carlo error}

	%{\todo If you can find the ergodic geometry somewhere (is it in the paper by Vollmer) then just cite it. We should avoid including so many different things in the paper}
	To analyze the MCMC error, we use the concept of `second largest eigenvalue' for Markov operators. We follow  the literature, e.g.,  \cite{Rudolf2012}. 
%%%%%%%%%%%%%%%%%%%%%%%%%%%%%%%%%%%%%%%%%%%%%%%%%%%%%%%%%%%%
	\begin{comment}
		Let $(X_n)_{n=1}^{\infty}$ be a Markov chain with transition kernel $K$ on some state space $S$. Let $\pi$ be a probability measure on $S$. The corresponding Markov operator $\mathcal{K}$ acting on $L^2_\pi(S,\mathbb{R})$ is defined by 
		\begin{equation*}
			\mathcal{K}f(u):=\int_S f(v)K(u,dv). 
		\end{equation*}
		If $K$ is reversible with respect to $\pi$, it is well-known that the linear operator $\mathcal{K}$ is self-adjoint and the spectrum of $\mathcal{K}$ lies in the interval $[-1,1]$. Moreover, $1$ is the largest eigenvalue of $\mathcal{K}$ since $\mathcal{K}c=c$ for any constant functional $c$. The orthogonal complement of the space of constant functionals is $L_0^2(\pi):=\{f\in L^2_\pi(S,\mathbb{R}):E^{\pi}[f]=0\}$. The `second largest eigenvalue' $\lambda_{(2)}^{K,\pi}$ is defined as the least upper bound of the spectrum of the restriction of $\mathcal{K}$ to $L_0^2(\pi)$, that is, 
		\begin{equation}\label{spectral_def}
			\begin{split}
				\lambda_{(2)}^{K,\pi}:=& \sup\{\alpha: \alpha\in \operatorname{spec}(\mathcal{K}|L_0^2(\pi))\}\\
				=&\sup_{f\in L_0^2(\pi), ||f||_{L^2(\pi)}=1}\langle Kf,f\rangle_{L^2(\pi)}.
			\end{split}
		\end{equation} 
		We remark that the `second largest eigenvalue' could be $1$. 
	\end{comment}
%%%%%%%%%%%%%%%%%%%%%%%%%%%%%%%%%%%%%%%%%%%%%%%%%%%%%%%%%%%%%%%%%%%%%	

 For the proposal $Q^J$ in \eqref{proposal_reversible}, 	
	we define the 'second largest eigenvalue' with respect to $Q^J$ and $\mu_0^J$ as follows. The corresponding Markov operator $\mathcal{Q}^J$ acting on  $L^2(U_J,\mathbb{R};\mu_0^J)$ is defined by 
	\begin{equation*}
		\mathcal{Q}^J f(\by):=\int_{U_J} f( \bs)Q^J(y,d\bs). 
	\end{equation*}
	Since $Q^J$ is reversible with respect to $\mu_0^J$, the linear operator $\mathcal{Q}^J$ is self-adjoint and its spectrum lies in the interval $[-1,1]$. Moreover, $1$ is the largest eigenvalue of $\mathcal{Q}^J$ since $\mathcal{Q}^J c=c$ for any constant functional $c$. Let $L_0^2(\mu_0^J)$ be the space of measurable functions $f:U_J\to \mathbb{R}$ such that 
	\begin{equation*}
		\int_{U_J}f(\by)d\mu_0^J(\by)=0; \quad \int_{U_J}[f(\by)]^2 d\mu_0^J(\by)=1.
	\end{equation*}
	The `second largest eigenvalue' $\lambda_{(2)}^{Q^J,\mu_0^J}$ is defined as
	\begin{equation}\label{spectral_def}
		\lambda_{(2)}^{Q^J,\mu_0^J}:= \sup_{f\in L_0^2(\mu_0^J)} \int_{U_J} [\mathcal{Q}^J f(\by)]f(\by)d\mu_0^J(\by).
	\end{equation}
	Equivalently, $\lambda_{(2)}^{Q^J,\mu_0^J}$ can be interpreted as the least upper bound of the spectrum of the restriction of $\mathcal{Q}^J$ to the orthogonal complement of the space of constant functionals \cite{Rudolf2012}. 
	
	Let $K^{J,l,N}$ be the transition kernel with proposal $Q^J$ and acceptance probability $\alpha^{J,l,N}$ in \eqref{eq:alphaJlN}. 
	Since $K^{J,l,N}$ is reversible with respect to the approximated posterior distribution $\mu^{J,l,N,\delta}$, we can define the `second largest eigenvalue' $\lambda_{(2)}^{K^{J,l,N},\mu^{J,l,N,\delta}}$ in a similar manner as \eqref{spectral_def} when $Q^J$ and $\mu_0^J$ are replaced by $K^{J,l,N}$ and $\mu^{J,l,N,\delta}$, respectively. 
	
	We assume that the second largest eigenvalue of $Q^J$ is uniformly bounded by a constant strictly less than $1$.
	\begin{assumption}\label{spectral_gap_assumption}
		$\sup_J \lambda_{(2)}^{Q^J,\mu_0^J}<1$.
	\end{assumption}
	This assumption can be verified for the Independence Sampler and Reflection Random Walk Metropolis  sampling methods. For Independence Sampler, %if we use $Q_{IS}^J$ as the transition kernel, the resulting Markov chain is just a sequence of i.i.d. samples from $\mu_0^J$. This corresponds to the Markov operator $\mathcal{Q}_{IS}^J(f):=E^{\mu_0^J}[f]$. Hence, 
it is clear from \eqref{spectral_def} that $\lambda_{(2)}^{Q_{IS}^J,\mu_0^J}=0$ for all $J$. For the Reflection Random Warl Metropolis, the above assumption is verified in \cite{Vollmer2015} in the case $\mu_0^J$ is the uniform prior on $U_J$.

	\begin{lemma}\label{spectral_gap_bound}
		Under Assumption \ref{spectral_gap_assumption}, there exists a positive constant $\lambda_{(2)}<1$ such that 
		\begin{equation*}
			\sup_{J,l,N} \lambda_{(2)}^{K^{J,l,N},\mu^{J,l,N,\delta}}\le \lambda_{(2)}. 
		\end{equation*}
	\end{lemma}
	\begin{proof}
		Since $|\mathcal{G}^{J,l,N}(\by)|$ is uniformly bounded, there exists a constant $c>0$ independent of $J,l,N$ such that for all $\by\in U_J$, 
		\begin{equation*}
			0<c\le \exp(-\Phi^{J,l,N}(\by;\delta))\le 1. 
		\end{equation*}
		Theorem 3.3 in \cite{Vollmer2015} implies that for any $J,l,N$, 
		\begin{equation*}
			1 - \lambda_{(2)}^{K^{J,l,N},\mu^{J,l,N,\delta}}\ge  c^4(1-\lambda_{(2)}^{Q^J,\mu_0^J}). 
		\end{equation*}
		By Assumption \ref{spectral_gap_assumption}, the RHS is uniformly bounded below by a positive constant. 
	\end{proof}
	
	For each probability distribution $\mu$ on $U_J$, we denote by $\mathcal{E}^{\mu,J,l,N}$ the expectation taken with respect to the joint distribution of the Markov chain $(u^{(i)})_{i=1}^{\infty}$ with transition kernel $K^{J,l,N}$ and initial distribution $\mu$. 
	From Corollary 3.27 in \cite{Rudolf2012} and Lemma \ref{spectral_gap_bound},  we deduce that if Assumption \ref{spectral_gap_assumption} holds  then for any function $g: U_J\to \mathbb{R}$ such that $\sup_{y\in U_J}|g(\by)|\le 1$,  
	\begin{equation}\label{MSE_Rudolf}
		\left(\mathcal{E}^{\mu^{J,l,N,\delta},J,l,N}\left[\left|E^{\mu^{J,l,N,\delta}}[g]-E_M^{\mu^{J,l,N,\delta}}[g]\right|^2\right]\right)^{1/2}\le \sqrt{\frac{2}{1-\lambda_{(2)}}} M^{-1/2}.
	\end{equation}
	
	\begin{comment}
			\begin{remark}
			In \cite{Rudolf2012}, there are results concerning the non-asymptotic error of burn-in. But to apply these results, we also need a nontrivial lower bound for $\lambda^{Q^J,\mu_0^J}_{min}:=\inf\{\alpha: \alpha \in \operatorname{spec}(Q^J|L_0^2(\mu_0^J))\}$. For IS, it is still 0. But for local Metropolis-Hastings methods like RRWM, the lower bound is not yet available in the literature. 
		\end{remark}
	\end{comment}

	\begin{theorem}\label{error_total}
		Under Assumptions \ref{lp-summability_assumption}, \ref{assumption:sigma,zeta}, \ref{lebesgue_const},   \ref{spectral_gap_assumption}, there exists a nested sequence of lower sets $(\Lambda_N)_{N\ge 1}$ such that $\#(\Lambda_N)=N$ and 
		\begin{equation}\label{error_eq}
			\left(\mathcal{E}^{\mu_0^J,J,l,N}\left[\left\Vert E^{\mu^{J,\delta}}[\sigma^J]- E_M^{\mu^{J,l,N,\delta}}[\sigma^J]\right\Vert_{L^\infty(D)}^2\right]\right)^{1/2}\le C(2^{-2l}+N^{-s}+M^{-1/2}), 
		\end{equation}
		where $C$ is a constant independent of $J,l,N,M$. 
	\end{theorem}
	\begin{proof}
		As $\sup_{\by\in U_J}||\sigma^J(\cdot,\by)||_{L^\infty(D)}\le \sigma^+$, applying Lemma 21 in \cite{Dashti2017} yields 
		\begin{equation*}
			||E^{\mu^{J,\delta}}[\sigma^J]-E^{\mu^{J,l,N\delta}}[\sigma^J]||_{L^\infty(D)}\le 2\sqrt{2}\sigma^{+}	d_{Hell}(\mu^{J,l,N,\delta},\mu^{J,\delta}).
		\end{equation*} 
		Combining with the Hellinger distance estimate in Proposition \ref{Hell_uniform_bound}, we have 
		\begin{equation}\label{error_Nl}
			||E^{\mu^{J,\delta}}[\sigma^J]-E^{\mu^{J,l,N\delta}}[\sigma^J]||_{L^\infty(D)}\le C_1(2^{-2l}+N^{-s}),  
		\end{equation}
		for some constant $C_1$ independent of $J,l,N$. 
		
		Due to the linearity of expectations,  
		\begin{equation*}
			E^{J,l,N,\delta}[\sigma^J]-E^{\mu^{J,l,N,\delta}}_M[\sigma^J] = \sum_{j=1}^J \psi_j(x) \left[E^{J,l,N,\delta}[ y_j]-E^{\mu^{J,l,N,\delta}}_M[ y_j]\right], 
		\end{equation*}
%		where $g_j:U_J\mapsto [-1,1]$ is defined by $g_j(y):=y_j$, $j=1,\dots, J$. 
Applying \eqref{MSE_Rudolf} for  $g=y_j$ and using the triangle inequality, we deduce 
		\begin{equation*}
			\left(\mathcal{E}^{\mu^{J,l,N,\delta},J,l,N}\left[\left|E^{\mu^{J,l,N,\delta}}[\sigma^J]-E_M^{\mu^{J,l,N,\delta}}[\sigma^J]\right|^2\right]\right)^{1/2}\le \sqrt{\frac{2}{1-\lambda_{(2)}}} M^{-1/2} \sum_{j=1}^{J}||\psi_j||_{L^\infty(D)}. 
		\end{equation*}  
		From Assumption \ref{lp-summability_assumption} about the decay rate of $(||\psi_j||_{L^\infty(D)})_{j\ge 1}$, we have 
		\begin{equation*}
			\sum_{j=1}^{J}||\psi_j||_{L^\infty(D)}\le \sum_{j=1}^{\infty}||\psi_j||_{L^\infty(D)}<\infty. 
		\end{equation*}
		On the other hand, note that 
		\begin{equation*}
			\frac{d\mu_0^J}{d\mu^{J,l,N,\delta}} = Z^{J,l,N}(\delta)\exp(\Phi^{J,l,N}(\by;\delta))\le \exp(\Phi^{J,l,N}(\by;\delta)).  
		\end{equation*} 
		Since $|\mathcal{G}^{J,l,N}(\by)|$ is uniformly bounded, the Radon-Nikodym derivative of $\mu_0^J$ with respect to $\mu^{J,l,N,\delta}$ is uniformly bounded. Hence, there exists a constant $C_2$ independent of $J,l,N,M$ such that 
		\begin{equation}\label{error_M}
			\left(\mathcal{E}^{\mu_0^{J},J,l,N}\left[\left|E^{\mu^{J,l,N,\delta}}[\sigma^J]-E_M^{\mu^{J,l,N,\delta}}[\sigma^J]\right|^2\right]\right)^{1/2}\le C_2 M^{-1/2}. 
		\end{equation} 
		The proof is completed by combining \eqref{error_Nl} and \eqref{error_M}.
	\end{proof}
	
	\begin{proof}[\textbf{Proof of Theorem \ref{error_total_J}}]
			Proposition \ref{mu-muJ} and Lemma 21 in \cite{Dashti2017} imply that 
		\begin{equation*}
			||E^{\mu^\delta}[\sigma]-E^{\mu^{J,\delta}}[\sigma]||_{L^\infty(D)} \le 2\sqrt{2} \sigma^+ d_{Hell}(\mu^\delta,\mu^{J,\delta})\le CJ^{-s}.  
		\end{equation*}
		On the other hand, from \eqref{sigma-sigmaJ}, 
		\begin{equation*}
			||E^{\mu^{J,\delta}}[\sigma]-E^{\mu^{J,\delta}}[\sigma^J]||_{L^\infty(D)}\le CJ^{-s}. 
		\end{equation*}
		Hence, by the triangle inequality,  
		\begin{equation}\label{error_trunc}
			||E^{\mu^\delta}[\sigma]-E^{\mu^{J,\delta}}[\sigma^J]||_{L^\infty(D)}\le CJ^{-s}. 
		\end{equation}
		Note that $E^{\mu^{J,\delta}}[\sigma^J]$ is identical whether $\mu^{J,\delta}$ is regarded as a measure on $[-1,1]^\mathbb{N}$ or on  $[-1,1]^J$. 
		Combining \eqref{error_trunc} and Theorem \ref{error_total}, we get the desired result. 
	\end{proof}

\end{document}